\documentclass[10pt,a4paper]{amsart}

\usepackage[latin1]{inputenc}

\usepackage{amssymb}
\usepackage{amsmath}
\usepackage{amsfonts}
\usepackage{amssymb}
\usepackage{amsthm}
\usepackage{enumerate}
\usepackage{calrsfs}

\newtheorem{theorem}{Theorem}[section]
\newtheorem{lemma}[theorem]{Lemma}
\newtheorem{proposition}[theorem]{Proposition}
\newtheorem{corollary}[theorem]{Corollary}

\newcommand{\ps}[2]{\langle\,#1,#2\,\rangle}

\newcommand{\R}{\mathbb{R}}
\newcommand{\N}{\mathbb{N}}
\newcommand{\C}{\mathbb{C}}

\newcommand{\Z}{\mathbb{Z}}

\newcommand{\cC}{{\mathcal C}}

\newcommand{\cF}{{\mathcal F}}   
   
\newcommand{\cH}{{\mathcal H}}

\newcommand{\cR}{{\mathcal R}}   
\newcommand{\cS}{{\mathcal S}}

\newcommand{\cW}{{\mathcal W}}

\newcommand{\mR}{{\mathbf R}}
\newcommand{\mK}{{\mathbf K}}

\newcommand{\dist}{\operatorname{dist}}
\newcommand{\supp}{\operatorname{supp}}
\newcommand{\weakto}{\rightharpoonup}
\newcommand{\weak}{\rightharpoonup}

\newcommand{\embed}{\hookrightarrow}

\newcommand{\eps}{\varepsilon}      
   
\newcommand{\ccW}{{\stackrel{}{{\mbox{\tiny$\cW$}}}}}

\renewcommand{\phi}{\varphi}   

\renewcommand{\Re}{\textrm{Re}}
\renewcommand{\Im}{\textrm{Im}}

\def\eps{\varepsilon}

\def\wh{\widehat}

\begin{document}

\title[Dual variational methods for the nonlinear Helmholtz equation]{Dual variational
methods and nonvanishing for the nonlinear Helmholtz equation}
\author{Gilles Evequoz}
\author{Tobias Weth}
\address{Institut f\"ur Mathematik, Johann Wolfgang Goethe-Universit\"at,
Robert-Mayer-Str. 10, 60054 Frankfurt am Main, Germany}
\email{evequoz@math.uni-frankfurt.de}
\email{weth@math.uni-frankfurt.de}

\begin{abstract}
We set up a dual variational framework to detect real standing wave solutions of the
nonlinear Helmholtz equation
$$
-\Delta u-k^2 u =Q(x)|u|^{p-2}u,\qquad u \in W^{2,p}(\R^N)
$$
with $N\geq 3$, $\frac{2(N+1)}{(N-1)}< p<\frac{2N}{N-2}$ and
nonnegative $Q \in
L^\infty(\R^N)$. We prove the existence of nontrivial solutions for
periodic $Q$ as well as in the case where $Q(x)\to 0$ as
$|x|\to\infty$. In the periodic case, a key ingredient of the approach
is a new nonvanishing
theorem related to an associated integral equation. The solutions we
study are superpositions of outgoing and incoming waves and are
characterized by a nonlinear far field relation. 
\end{abstract}


\maketitle

\section{Introduction}
\label{sec:introduction}

Due to their importance in various problems in physics, nonlinear stationary Schrödinger
equations of the type
\begin{equation}
\label{eq:27}
- \Delta u + \lambda u = f(x,u),\quad x\in \R^N 
\end{equation}
have been studied extensively since the pioneering works of Berestycki
and Lions \cite{berestycki-lions83,berestycki-lions83b}, Lions
\cite{lions1}, Floer and Weinstein
\cite{floer.weinstein:86}, Ding and Ni \cite{ding.ni:86}
and Rabinowitz \cite{rabinowitz:92} from the 1980s and 1990s. For
superlinear nonlinearities of the form $f(x,u)= r(x,|u|^2)u$, solutions of (\ref{eq:27}) correspond to periodic
solutions of the time-dependent nonlinear  
Schr{\"o}dinger equation 
$$
i \partial_t \psi(t,x) = -\Delta \psi(t,x) 
-f(x,\psi(t,x)),\qquad  (t,x) \in \R \times \R^N.
$$
via the ansatz $\psi(t,x)= e^{ i\lambda
  t}u(x)$. Moreover, for $m \ge 0$ and $\lambda<m$, the ansatz 
$\psi(t,x)=e^{i \sqrt{m-\lambda} t}u(x)$ leads to periodic solutions of 
the 
nonlinear Klein-Gordon equation 
\begin{equation}
\label{eq:29}
\frac{\partial^2 \psi}{\partial t^2}(t,x) - \Delta \psi(t,x) + m \psi(t,x)  =f(x,\psi(t,x)), \qquad (t,x)\in\R\times\R^N.
\end{equation}
In the present paper we are interested in standing wave solutions to
\eqref{eq:29} which arise from real-valued solutions $u$ of \eqref{eq:27}. We note that real-valued solutions 
of \eqref{eq:27} with the decay property $u(x) \to 0$ as $|x| \to \infty$ have
been studied extensively in the case where $\partial_u f(\cdot ,0)
\equiv 0$ on
$\R^N$ and $\lambda \ge 0$, see e.g. \cite{kuzin-pohozaev,struwe,willem} and the references
therein. On the contrary, very little is known in the case
$\lambda<0$, where $0$ is contained in the essential spectrum of the
Schr\"odinger operator $-\Delta +\lambda$. In the present paper we are
interested in this case, which is relevant for the analysis of standing wave solutions of
\eqref{eq:29} with large frequencies according to the ansatz above. In
this case, it is customary to set
$\lambda=-k^2$, and (\ref{eq:27}) is called
nonlinear Helmholtz equation (or nonlinear reduced wave
equation). By restricting our attention to the important class of
power type nonlinearities $f(x,u)= Q(x)|u|^{p-2}u$, we are therefore
led to study real-valued solutions of the problem 
\begin{equation}\label{eqn:1}
-\Delta u -k^2 u =Q(x)|u|^{p-2}u, \qquad u \in W^{2,p}(\R^N).
\end{equation}
One of the very few existence results available for \eqref{eqn:1} is due
to Guti\'errez \cite{gutierrez04}, who studied the special case
$N=3,4$, $p=4$, $Q \equiv \pm 1$. She proved the
 existence of small complex solutions of \eqref{eqn:1} with the
 additional (finiteness) property 
 \begin{equation}
   \label{eq:30}
\sup_{R>1}\frac{1}{R} \int_{B_R}|u|^2\,dx < \infty.   
 \end{equation}
Here and in the following, $B_R \subset \R^N$ denotes the open ball of
radius $R$ centered around the origin. In order to give a more precise
description of the solutions obtained by Guti\'errez, we briefly recall some important facts
on \eqref{eqn:1} in the case $N \ge 3$, $Q \in L^\infty(\R^N)$ and
$\frac{2(N+1)}{(N-1)} \le p \le \frac{2N}{N-2}$. In this case, a (complex-valued) function $u \in W^{2,p}(\R^N)$ solves
\eqref{eqn:1},~\eqref{eq:30} if and only if $u \in L^p(\R^N)$ solves
\begin{equation}
\label{eq:2}
u = \cR\Bigl(Q(x)|u|^{p-2}u\Bigr)+\phi_g \qquad \text{for some
  $\phi_g \in \cH$.} 
\end{equation}
Here $\cR$ is the resolvent operator given by convolution with the
fundamental solution of the linear Helmholtz equation (see
Section~\ref{sec:resolvent} below), and $\cH$ denotes the space of Herglotz wave functions 
$\phi_g: \R^N \to \C$ given as 
$$
\phi_g(x)= \int_{S^{N-1}} e^{i k(x\cdot \xi)} g (\xi)\,d \sigma(\xi) \qquad
\text{for some function $g \in L^2(S^{N-1})$.}
$$
It follows from the dual version of the Stein-Tomas-Theorem (see
Theorem~\ref{sec:resolv-estim-nonv} below) and elliptic estimates that
such functions satisfy $\phi_g \in W^{2,p}(\R^N)$ and solve the linear
Helmholtz equation $(\Delta +k^2) \phi_g = 0$ in the strong sense. For a given solution $u$ of
\eqref{eqn:1},~\eqref{eq:30}, the functions $g$ and $\phi_g$ are uniquely
determined by \eqref{eq:2}, and we will call $\phi_g$ the
{\em Herglotz wave associated to $u$} in the sequel.

In the case $N=3,4$, $p=4$, $Q \equiv \pm 1$, it was proved 
in \cite[Theorem 1]{gutierrez04} that for given small $g\in L^2(S^{N-1})$ the 
problem \eqref{eq:2} admits a unique (complex-valued) solution $u \in W^{2,4}(\R^N)$ 
which is also small in the $L^4$-norm. 
The proof is based on subtle resolvent estimates combined with a
contraction mapping argument. As mentioned in \cite[Page 3]{gutierrez04},
the same argument also gives small real-valued solutions of \eqref{eq:2}, when replacing 
the resolvent operator $\cR$ by its real part and considering real-valued functions 
$g\in L^2(S^{N-1})$ such that $g(-\xi)=g(\xi)$ for all $\xi\in S^{N-1}$.
In the present paper, we focus on a
complementary class of real-valued solutions of \eqref{eqn:1} which
satisfy the integral equation 
\begin{equation}
\label{eq:28}
u = \mR\Bigl(Q(x)|u|^{p-2}u\Bigr), \qquad u \in L^p(\R^N),
\end{equation}
where $\mR$ denotes the real part of the resolvent operator
$\cR$. We shall see that these solutions satisfy \eqref{eq:2} with 
\begin{equation}
  \label{eq:32}
g_u(\xi)= -\frac{i}{4}\Bigl(\frac{k^2}{2\pi}\Bigr)^{\frac{N-2}{2}}\cF(Q|u|^{p-2}u)(k\xi), \quad \xi\in S^{N-1},
\end{equation}
where, here and in the following, $\cF$ denotes the Fourier
transform (see \eqref{eqn:g} below). So the associated Herglotz wave 
$\phi_{g_u}$ is related in a nonlinear way to the solution $u$ itself. 
From resolvent estimates which we
recall in Section~\ref{sec:resolvent}, it easily follows that $u
\equiv 0$ is an isolated solution of \eqref{eq:28} in
$L^p(\R^N)$, and thus nontrivial solutions cannot be found by a
contraction mapping argument as in \cite{gutierrez04}. In this paper,
we set up a variational framework to find nontrivial solutions of this problem. Before stating our main
results, we mention that there is an intimate
relationship between the far field behavior of solutions of
\eqref{eqn:1},~\eqref{eq:30} and their associated Herglotz waves. In
particular, we shall see that if $u \in L^{p}(\R^N)$
solves \eqref{eq:28}, then not only does it solve the problem \eqref{eqn:1}
but it also satisfies the far field relation
\begin{equation}\label{eqn:weak_asympt-r-introduction}
\lim_{R\to\infty}\frac{1}{R}\int\limits_{B_R}\Bigl|u(x)
+2\Bigl(\frac{2\pi}{k|x|}\Bigr)^{\frac{N-1}{2}} \Re \bigl[e^{ik|x|-\frac{i(N-1)\pi}{4}}g_u(\hat x)\bigr]\Bigr|^2\, dx=0,
\end{equation}
with $\hat x= \frac{x}{|x|}$ for $x \in \R^N \setminus \{0\}$ and $g_u$ as
in \eqref{eq:32}. Note that \eqref{eqn:weak_asympt-r-introduction}
implies \eqref{eq:30}. We are now in a position to state our main
results which are related to two different types of weight functions $Q$.

\begin{theorem}
\label{main-theorem-periodic}
Let $N \ge 3$, $\frac{2(N+1)}{N-1} <p< \frac{2N}{N-2}$, and let $Q \in
L^\infty(\R^N)$, $Q \ge 0$, $Q \not \equiv 0$ be $\Z^N$-periodic. Then
problem \eqref{eqn:1}, \eqref{eqn:weak_asympt-r-introduction} admits a
nontrivial strong solution such that $u \in W^{2,q}(\R^N) \cap
\cC^{1,\alpha}(\R^N)$ for all $q \in [p,\infty)$, $\alpha \in (0,1)$. 
\end{theorem}

\begin{theorem}
\label{main-theorem-compact}
Let $N \ge 3$, $\frac{2(N+1)}{N-1} \leq p< \frac{2N}{N-2}$, and let $Q \in
L^\infty(\R^N)$, $Q \ge 0$, $Q \not \equiv 0$ satisfy $\lim
\limits_{|x| \to \infty}Q(x)=0$. Then
problem \eqref{eqn:1}, \eqref{eqn:weak_asympt-r-introduction} admits a
sequence of pairs $\pm u_n$ of solutions such that $u_n \in W^{2,q}(\R^N) \cap
\cC^{1,\alpha}(\R^N)$ for all $q \in [p,\infty)$, $\alpha \in (0,1)$,
and 
\begin{equation}
  \label{eq:26}
\|u_n\|_{L^p(\R^N)} \to \infty \qquad \text{as $n \to \infty$.}  
\end{equation}
\end{theorem}

Up to our knowledge, these results are the first existence results for
problem \eqref{eqn:1},~\eqref{eqn:weak_asympt-r-introduction} under
the given assumptions on $Q$. The main difficulty of the problem is
the lack of a direct variational approach, since the energy functional
formally associated to \eqref{eqn:1} is not well defined
on $W^{2,p}(\R^N)$, and it is not even well defined on nontrivial solutions
of \eqref{eqn:1}, \eqref{eqn:weak_asympt-r-introduction}. 
In a previous paper by the authors \cite{evequoz-weth14}, the
case of compactly supported $Q$ (and a more general class of superlinear
nonlinearities $f$ compactly supported in space) has been studied with
a variational reduction method. More precisely, in
\cite{evequoz-weth14} we used a Dirichlet-to-Neumann map associated to the exterior problem for the linear
Helmholtz equation to reduce the problem to the existence of nontrivial critical points of an energy functional in $H^1(B_R(0))$ 
for some $R>0$. We then used linking arguments to get existence results. The method of \cite{evequoz-weth14} is obviously
restricted to (spatially) compactly supported nonlinearities and
therefore cannot be used to derive the results of the present
paper. 
On the other hand, in the case of
compactly supported $Q$, multiple existence of solutions can be shown
for any $N \ge 1$ and for a larger range of exponents, namely for $p>2$ with $p<\frac{2N}{N-2}$ if $N \ge 3$, see
\cite{evequoz-weth14}. 

It is natural to ask whether a stronger, pointwise version
of the far field relation \eqref{eqn:weak_asympt-r-introduction} in the form
\begin{equation}
  \label{eq:35}
 u(x)= -2\Bigl(\frac{2\pi}{k|x|}\Bigr)^{\frac{N-1}{2}} \Re \bigl[e^{ik|x|-\frac{i(N-1)\pi}{4}}g_u(\hat x)\bigr]
+o(|x|^{\frac{1-N}{2}}) \qquad \text{as $|x| \to \infty$.}  
\end{equation}
is
available for the solutions given by Theorems~\ref{main-theorem-periodic} and
\ref{main-theorem-compact} above. Related to this question, we have
the following result. 

\begin{theorem}
\label{sec:introduction-4}
Let $N=3$, $4<p\leq 6$ or $N=4$, $\frac{11}{3}<p\leq 4$, and let $Q \in
L^\infty(\R^N)$. Then every solution $u \in L^p(\R^N)$ of
\eqref{eq:28} satisfies \eqref{eq:35}. In particular, $u$ has
pointwise decay given by $|u(x)| = O(|x|^{\frac{1-N}{2}})$ as $|x| \to \infty$. 
\end{theorem}

Note that, for $N=3$, the full range of admissible exponents in
Theorem~\ref{main-theorem-periodic} is covered by
Theorem~\ref{sec:introduction-4}, and thus \eqref{eq:35} holds for the
solution detected in this Theorem. Moreover, for $N=3$, \eqref{eq:35}
also holds for the
solutions detected in Theorem~\ref{main-theorem-compact} unless $p=4$
in which case the question is open. 
The first step in the proof of Theorem~\ref{sec:introduction-4} is to show that
$u \in L^{p-1}(\R^N)$, which implies that the RHS of (\ref{eqn:1}) is in $L^1(\R^N)$. This is done by a bootstrap argument, using
$L^p - L^q$ mapping properties of the resolvent due to Guti\'errez \cite[Theorem 6]{gutierrez04}
for exponents $p$ and $q$ lying off the line of duality $p=q'$.
It is open whether the restrictions
on $N$ and $p$ in Theorem~\ref{sec:introduction-4} are necessary, but
we note that for $p<\frac{3N-1}{N-1}$ the property \eqref{eq:35}
implies in general that $u \not \in L^{p-1}(\R^N)$.

Let us now briefly explain our approach and the organization of the
paper. In Section~\ref{sec:resolvent}, we first recall important estimates
and characterizations related to
the linear (homogeneous and inhomogeneous) Helmholtz equation due to
Kenig, Ruiz and Sogge \cite{KRS87}, Guti\'errez \cite{gutierrez04} and 
Agmon \cite{agmon90,agmon99}. In particular, we recall some mapping
properties of the resolvent $\cR$ with respect to Lebesgue spaces.
Moreover, we derive far field asymptotics
within the linear inhomogeneous setting. In
Section~\ref{sec:nonv-prop}, we then derive a nonvanishing property
related to the resolvent which is a key ingredient in the proof of
Theorem~\ref{main-theorem-periodic}. With the help of this
nonvanishing property, the problem of lack of compactness of the periodic case
will be overcome. The nonvanishing property can be seen as an
analogue of Lions' local compactness Lemma (see e.g. \cite{lions1} or \cite[Lemma
1.21]{willem}), and it relies, in particular, on a combination of arguments as in
\cite{alama-li92} with an asymptotic multiplier estimate inspired by \cite{gutierrez04}. 
In Section~\ref{sec:dual} we set up a dual variational framework for
problem \eqref{eqn:1},~\eqref{eqn:weak_asympt-r-introduction} which
relies on the corresponding integral equation~\eqref{eq:28}. More
precisely, we define an energy functional on $L^{p'}(\R^N)$
such that \eqref{eq:28} is reformulated as the corresponding Euler-Lagrange
equation for $v:= Q^{\frac{1}{p'}}|u|^{p-2}u$. Here $p':= \frac{p}{p-1}$ denotes the
conjugate exponent of $p$. This approach is inspired, in particular, by \cite{alama-li92,jeanjean94}.
We also complete the proof of Theorem~\ref{sec:introduction-4} in this
section. In Section \ref{sec:compact} we then consider the case where
the coefficient $Q$ satisfies $Q(x) \to 0$ as $|x| \to \infty$, and
we show that in this case the dual energy functional satisfies the
Palais-Smale condition. We then apply a variant of the symmetric
mountain-pass lemma in order to complete the proof of
Theorem~\ref{main-theorem-compact}. In Section~\ref{sec:dual_periodic}
we consider the case of periodic $Q$. In this case, the dual energy functional
also has the mountain-pass geometry, but it does not satisfy the
Palais-Smale condition anymore. Nevertheless, the existence of a
bounded Palais-Smale sequence can be shown, and it remains to show
that a subsequence converges strongly in $L^{p'}_{\text{loc}}(\R^N)$ to a
nontrivial critical point of the functional. By this we complete the
proof of Theorem~\ref{main-theorem-periodic}. Finally, in Appendix 
\ref{sec:appendix}, we add a result, based on standard elliptic
estimates, on the Sobolev regularity imposed by the
resolvent operator $\cR$. 

We close this introduction by fixing some notation. Throughout the
paper, we let $B_R(x)$ denote the open ball of radius $R$ centered at
$x$, and we also set $B_R:=B_R(0)$ and $M_R= \R^N \setminus B_R$. 
As already mentioned, we put $\wh{x} = \frac{x}{|x|}$ for $x \in \R^N
\setminus \{0\}$. The symbols $\cS$ and $\cS'$ respectively denote the Schwartz space
and the space of tempered distributions on $\R^N$. For
$f \in \cS'$, we write $\cF(f)$ or $\widehat f$ to denote the Fourier
transform of $f$. Moreover, for matters of simplicity, we sometimes write $\|\cdot\|_s$ instead
of $\|\cdot\|_{L^s(\R^N)}$ for $s \in [1,\infty]$.

Throughout the remainder of the paper, we restrict our attention to
the case $k=1$. This leads to less complicated formulas in the
derivations, and the general case follows from the scaling properties
of \eqref{eqn:1} and the linear (homogeneous and inhomogeneous)
Helmholtz equation.

\section{Resolvent estimates and far field asymptotics}\label{sec:resolvent}
Throughout this section, we regard all function spaces as spaces of
complex-valued functions. Let $\eps>0$. Then the operator $-\Delta-(1+i\eps)$: $H^2(\R^N)\subset L^2(\R^N)$
$\to$ $L^2(\R^N)$ is an isomorphism. Moreover, for any $f$ from the Schwartz space $\cS$ its inverse is given by
$$
\cR_\eps f(x):=[-\Delta-(1+i\eps)]^{-1}f(x) = (2\pi)^{-\frac{N}{2}}
\int_{\R^N}e^{ix\cdot \xi}\frac{\wh{f}(\xi)}{|\xi|^2-(1+i\eps)}\, d\xi.
$$
It is well known (see e.g. \cite{gelfand}) that there exists a linear
operator $\cR : \cS \to \cS'$ given by    
$$
\langle \cR f, g \rangle := \lim_{\eps \to 0} \int_{\R^N} [\cR_\eps
f](x) g(x)\,dx =  \int_{\R^N} [\Phi * f](x) g(x)\,dx \qquad \text{for $f,g \in \cS$}
$$
with 
\begin{equation}\label{eqn:fund_sol}
\Phi(x):= (2\pi)^{-\frac{N}{2}} \cF^{-1}((|\xi|^2-1-i0)^{-1})(x)
=\frac{i}{4} (2\pi |x|)^{\frac{2-N}{2}}H^{(1)}_{\frac{N-2}{2}}(|x|)
\end{equation}
for $x \in \R^N \setminus \{0\}$, where $H^{(1)}_{\frac{N-2}{2}}$ is the Hankel function of the first kind of order $\frac{N-2}{2}$. 
Here we use the notation from \cite{gelfand}, which also allows us to briefly write 
$$
\cR f:=\cF^{-1}\left((|\xi|^2-1-i0)^{-1}\wh{f}\right)\qquad \text{for $f \in \cS$.}
$$
For $H^{(1)}_{\frac{N-2}{2}}$ we have the asymptotic expansions 
\begin{equation}
  \label{eq:3}
H^{(1)}_{\frac{N-2}{2}}(s) =\left \{
  \begin{aligned}
&\sqrt{\frac{2}{\pi s}}\,
  e^{i(s-\frac{N-1}{4}\pi)}[1+O(s^{-1})] &&\qquad \text{as $s \to
    \infty$,}\\
&-\frac{i\, \Gamma(\frac{N-2}{2})}{\pi}
\Bigl(\frac{2}{s}\Bigr)^{\frac{N-2}{2}}[1+O(s)]  &&\qquad \text{as $s \to 0^+$}
  \end{aligned}
\right.
\end{equation}
(see e.g. \cite[Formulas (5.16.3)]{lebedev}), so there exists a constant $C_0>0$ such that 
\begin{equation}
  \label{eq:4}
  |\Phi(x)| \le C_0 \max \{|x|^{2-N},|x|^{\frac{1-N}{2}}\} \qquad
  \text{for $x \in \R^N \setminus \{0\}$.}   
\end{equation}
Moreover, $\Phi$ satisfies the equation $-\Delta \Phi - \Phi = \delta$
together with Sommerfeld's outgoing 
radiation condition 
\begin{equation}
  \label{eq:24}
|\nabla \Phi(x)- i \Phi(x)\hat x| = o(|x|^{\frac{1-N}{2}}) \qquad
\text{as $|x| \to \infty$.}   
\end{equation}
As a consequence, for $f \in \cS$, the function 
$u=\cR f \in \cC^\infty(\R^N)$ is a solution of the inhomogeneous
Helmholtz equation $-\Delta u - u=f$ satisfying $|u(x)|=
O(|x|^{\frac{1-N}{2}})$ as $|x| \to \infty$ and the outgoing radiation
condition (\ref{eq:24}) with $u$ in place of $\Phi$. We also have the
following important estimates. 

\begin{theorem} {(Special case of \cite[Theorem 2.3]{KRS87} by Kenig, Ruiz and Sogge)}\label{thm:KRS}\\
 Let $\frac{2(N+1)}{N-1}\leq p\leq\frac{2N}{N-2}$.
 Then there exists a constant $C>0$ such that
 \begin{equation}
\label{eq:1}
  \|\cR f\|_{L^p(\R^N)}\leq C \|f\|_{L^{p'}(\R^N)} \quad\text{for all }f\in\cS.
 \end{equation}
\end{theorem}

As a consequence of this estimate, the operator 
$\cR$ can be continuously extended as a mapping from $L^{p'}(\R^N)$
into $L^p(\R^N)$ for $\frac{2(N+1)}{(N-1)}\leq p\leq\frac{2N}{N-2}$
such that (\ref{eq:1}) still holds for $f \in L^{p'}(\R^N)$. We note that an extension of Theorem~\ref{thm:KRS} to more general pairs of 
exponents has been obtained by Guti\'errez in \cite[Theorem 6]{gutierrez04}, see also the proof of Theorem~\ref{sec:nonv-thm-2} below.  
We also recall the following estimate from \cite{gutierrez04}.

\begin{theorem} {(Limit case of \cite[Theorem 7]{gutierrez04} by Guti\'errez)}\label{thm:gut7}\\
Let $\frac{2(N+1)}{N-1}\leq p\leq \frac{2N}{N-2}$. Then there exists a constant $C>0$ such that
\begin{equation}
\sup\limits_{R\geq 1} \frac1R\int\limits_{B_R}\bigl|\cR f(x)\bigr|^2\, dx\leq C \|f\|_{L^{p'}(\R^N)}^2\quad\text{for all }f\in\cS.
\end{equation}
\end{theorem}

Next we recall the
Stein-Tomas-Theorem which is fundamental for the study of the operator
$\cR$.  

\begin{theorem}{(Stein-Tomas-Theorem \cite{tomas75})}\\
\label{sec:resolv-estim-nonv}
Let $1 \le q \le \frac{2(N+1)}{N+3}$. Then there exists a constant
$C=C(q)>0$ such that for every $u \in L^q(\R^N)$ we have $\hat
u|_{S^{N-1}} \in L^2(S^{N-1})$ and 
$$
\int_{S^{N-1}}|\hat u(\theta)|^2\,d\theta \le C \|u\|_{L^q(\R^N)}^2.
$$
\end{theorem}

For general solutions of $-\Delta u-u=f$ in $\R^N$ we will show that some kind of asymptotic expansion holds. 
Let us first recall the following consequence of results by Agmon.

\begin{theorem}{(Agmon {\cite{agmon90,agmon99}})}\\ \label{thm:agmon}
Let $\frac{2N}{N-1}< p\leq\frac{2N}{N-2}$ and consider a (distributional) solution $w\in L^p(\R^N)$ of
$$
-\Delta w - w=0\quad\text{in }\R^N.
$$
Then there is a unique $g\in H^{-s}(S^{N-1})$, $s>\frac{N-1}{2}-\frac{N}{p}$, such that 
\begin{equation}\label{eqn:Herglotz}
	w(x)=\int_{S^{N-1}}e^{ix\cdot\xi}g(\xi)\, d\sigma(\xi),\quad x\in\R^N.
\end{equation}
If, in addition, $\sup\limits_{R>1}\frac1R\int_{B_R}|w|^2\, dx<\infty$, then $g\in L^2(S^{N-1})$, and
the following asymptotic expansion is valid:
\begin{equation}\label{eqn:expansion_inhom}
\begin{aligned}
	\lim_{R\to\infty}\frac1R\int\limits_{B_R}\Bigl|w(x)\!-\!\Bigl(\frac{2\pi}{|x|}\Bigr)^{\frac{N-1}{2}}\bigl[&e^{i|x|-\frac{i(N-1)\pi}{4}}g(\wh{x})
	+e^{-i|x|+\frac{i(N-1)\pi}{4}}g(-\wh{x})\bigr]\Bigr|^2 dx=0
\end{aligned}
\end{equation}
\end{theorem}
\begin{proof}
The first assertion follows from Theorem 4.1 and Theorem 6.2~(ii)(a) in \cite{agmon99}, whereas
the second assertion can be deduced from Theorem 4.3 and Theorem 4.5~(ii) in \cite{agmon90}.
\end{proof}
Remark that for $u\in L^p(\R^N)$, $f\in L^{p'}(\R^N)$ satisfying $-\Delta u-u=f$ in $\R^N$, the function $w=u-\cR f$ belongs to 
$L^p(\R^N)$, by Theorem~\ref{thm:KRS}, and solves $-\Delta w-w=0$ in $\R^N$. Hence, we have the following
\begin{corollary}\label{coroll:asympt_inhom_phi}
Let $u\in L^p(\R^N)$ be a solution of $-\Delta u -u=f$ in $\R^N$ with $f\in L^{p'}(\R^N)$ and 
$\frac{2(N+1)}{N-1}\leq p\leq \frac{2N}{N-2}$. 
Then there is a unique $g\in H^{-s}(S^{N-1})$, $s>\frac{N-1}{2}-\frac{N}{p}$, such that
\begin{equation}\label{eqn:agmon_inhom}
	u(x)=[\cR f](x)+\int_{S^{N-1}}e^{ix\cdot\xi}g(\xi)\, d\sigma(\xi),\quad \text{for a.e. }x\in\R^N.
\end{equation}
Moreover, if $\sup\limits_{R>1}\frac1R\int_{B_R}|u|^2\, dx<\infty$, then $g\in L^2(S^{N-1})$.
\end{corollary}
In order to obtain an asymptotic expansion for $u$, it remains to study the asymptotics 
of $\cR f$ where $f\in L^{p'}(\R^N)$. The remainder of this section is therefore devoted to the far field pattern
associated with the operator $\cR$. We start with the following result which is well-known
to experts, see e.g. \cite[Theorem 2.5]{colton-kress} for the case
$N=3$. Since we could not find a reference for general $N \ge 3$,
we give the proof here for the reader's convenience.

\begin{proposition}
\label{prop:farfield}
Let $f\in \cC_c^\infty(\R^N)$. Then,
$$
[\cR f](x)=
\sqrt{\frac{\pi}{2}}\ \frac{e^{i|x|-\frac{i(N-3)\pi}{4}}}{|x|^{\frac{N-1}{2}}}\
\wh{f}(\wh{x}) + O\left(|x|^{-\frac{N+1}{2}}\right) \qquad \text{as $|x|\to\infty$.}
$$
\end{proposition}
\begin{proof}
Choose $R>0$ such that $\supp f\subset B_R$. For $x\in\R^N$
with $|x|\geq 2R$ we can write, using the asymptotics in \eqref{eq:3}, 
\begin{align*}
[\cR f](x)=\gamma_N
\int\limits_{B_R}\frac{e^{i|x-y|}}{|x-y|^{\frac{N-1}{2}}}\Bigl(1+\delta(|x-y|)\Bigr)f(y)\,
dy, \quad\text{ $\gamma_N=\frac{(2\pi)^{\frac{1-N}{2}}}{2}e^{-\frac{i(N-3)\pi}{4}}$,}
\end{align*}
where the function $r \mapsto \delta(r)$ satisfies $\delta_*:= \sup\limits_{r\geq 1}r|\delta(r)|<\infty$.
Furthermore, there exists a constant $\zeta>0$ such that
\begin{equation}\label{eqn:11}
\Bigl|\ |x-y|-|x|+\wh{x}\cdot y\ \Bigr|\leq \frac{\zeta|y|^2}{|x|}\quad\text{ and }\quad
\left||x-y|^{-\frac{N-1}{2}}-|x|^{-\frac{N-1}{2}}\right|\leq \frac{\zeta|y|}{|x|^{\frac{N+1}{2}}}
\end{equation}
for all $x, y\in\R^N$ with $x\neq 0$ and $|y|\leq\frac{|x|}{2}$. 
As a consequence, we may estimate, for $|x|\geq 2R$,
\begin{align*}
&\Bigl|\int_{B_R}\Bigl(\frac{e^{i|x-y|}}{|x-y|^{\frac{N-1}{2}}}-\frac{e^{i|x|-i\wh{x}\cdot y}}{|x|^{\frac{N-1}{2}}}\Bigr)f(y)\, dy\Bigr|
\leq \int_{B_R}\Bigl|\frac{e^{i(|x-y|-|x|+\wh{x}\cdot y)}}{|x-y|^{\frac{N-1}{2}}}-\frac{1}{|x|^{\frac{N-1}{2}}}\Bigr| |f(y)|\, dy\\
&\leq \int\limits_{B_R}\Bigl(\frac{|e^{i(|x-y|-|x|+\wh{x}\cdot y)}-1|}{|x-y|^{\frac{N-1}{2}}}+\Bigl|\frac{1}{|x-y|^{\frac{N-1}{2}}}-\frac{1}{|x|^{\frac{N-1}{2}}}\Bigr|\Bigr)\ |f(y)|\, dy\\
&\leq \int_{B_R}\Bigl(\frac{\sqrt{2}\zeta |y|^2}{|x|\
  |x-y|^{\frac{N-1}{2}}}+\frac{\zeta|y|}{|x|^{\frac{N+1}{2}}}\Bigr)\
|f(y)|\, dy \leq \frac{\kappa_R \|f\|_{L^1(\R^N)}}{|x|^{\frac{(N+1)}{2}}}
\end{align*}
with $\kappa_R:=
(2^{\frac{N-1}{2}}\sqrt{2}R^2+R)\zeta$ and 
\begin{align*}
\Bigl|\,\int\limits_{B_R}\frac{e^{i|x-y|}}{|x-y|^{\frac{N-1}{2}}}\delta(|x-y|)f(y)\, dy\Bigr|
\leq\delta_*\int\limits_{B_R}\frac{|f(y)|}{|x-y|^{\frac{(N+1)}{2}}}\, dy
\leq \frac{2^{\frac{(N+1)}{2}}\delta_*}{|x|^{\frac{(N+1)}{2}}}\|f\|_{L^1(\R^N)}.
\end{align*}
Combining these two estimates, we obtain
\begin{align*}
\left|[\cR
  f](x)-\gamma_N\frac{e^{i|x|}}{|x|^{\frac{N-1}{2}}}(2\pi)^{\frac{N}{2}}\wh{f}(\wh{x})\right|
\quad\leq |\gamma_N|\left(\kappa_R+2^{\frac{(N+1)}{2}}\delta_*\right)\|f\|_{L^1(\R^N)} |x|^{-\frac{(N+1)}{2}}
\end{align*}
for $|x| \ge 2R$, and the conclusion follows.
\end{proof}

The pointwise asymptotic expansion given in Proposition~\ref{prop:farfield} does not extend to general functions $f \in L^{p'}(\R^N)$ in the case
where $\frac{2(N+1)}{(N-1)}\leq p\leq \frac{2N}{N-2}$. Nevertheless,
we have the following weaker variant of these asymptotics.

\begin{proposition}\label{prop:weak_farfield}
Let $\frac{2(N+1)}{N-1}\leq p\leq \frac{2N}{N-2}$ and $f\in L^{p'}(\R^N)$. Then
\begin{equation}\label{eqn:weak_asympt}
\lim_{R\to\infty}\frac{1}{R}\int\limits_{B_R}\Bigl|\cR
  f(x)-\sqrt{\frac{\pi}{2}}\ \frac{e^{i|x|-\frac{i(N-3)\pi}{4}}}{|x|^{\frac{N-1}{2}}}\
  \wh{f}(\wh{x})\Bigr|^2\, dx=0.
\end{equation}
\end{proposition}
\begin{proof}
As a consequence of Theorems~\ref{thm:KRS} and~\ref{sec:resolv-estim-nonv}, the integrand
in \eqref{eqn:weak_asympt} belongs to $L^2_{\text{loc}}(\R^N)$. Letting $\kappa_N=\sqrt{\frac{\pi}{2}}e^{-\frac{i(N-3)\pi}{4}}$ and
$w(x)=\cR f(x)-\kappa_N\frac{e^{i|x|}}{|x|^{\frac{N-1}{2}}}\
\wh{f}(\wh{x})$, it follows that $w\in L^2(B_1(0))$ and therefore
\begin{equation*}
\frac{1}{R}\int\limits_{\{|x|<1\}}\Bigl|\cR
  f(x)-\kappa_N\frac{e^{i|x|}}{|x|^{\frac{N-1}{2}}}\ \wh{f}(\wh{x})\Bigr|^2\, dx=\frac1R\|w\|_{L^2(B_1(0))}^2\to 0\quad\text{as $R\to\infty$.}
\end{equation*}
Let now $\eps>0$ be given, and consider $g\in \cC^\infty_c(\R^N)$ such
that $\|g-f\|_{L^{p'}(\R^N)}<\eps$. By Proposition \ref{prop:farfield}
we have 
$$
M_g:=\sup_{|x|\geq 1}|x|^{N+1}\Bigl|\cR
  g(x)-\kappa_N\frac{e^{i|x|}}{|x|^{\frac{N-1}{2}}}\ \wh{g}(\wh{x})\Bigr|^2<\infty
$$
and therefore 
\begin{align*}
\frac{1}{R}\int\limits_{\{1\leq|x|<R\}}\Bigl|\cR
  g(x)&-\kappa_N\frac{e^{i|x|}}{|x|^{\frac{N-1}{2}}}\ \wh{g}(\wh{x})\Bigr|^2\, dx\\
&\leq \frac{M_g}{R}\int\limits_{\{1\leq|x|<R\}} |x|^{-(N+1)}\, dx\leq \frac{M_g\omega_N}{R}\to 0,
\end{align*}
as $R\to\infty$. According to Theorem \ref{thm:gut7}, 
\begin{equation*}
\sup_{R \ge 1} \frac{1}{R}\int_{B_R}\Bigl|\cR f(x)-\cR g(x)\Bigr|^2\, dx< C \eps^2
\end{equation*}
with a constant $C>0$ independent of $g$. Moreover, by Theorem~\ref{sec:resolv-estim-nonv} we have 
\begin{align*}
\sup_{R \ge 1} \frac{1}{R}\int\limits_{B_R}\Bigl|\frac{e^{i|x|}}{|x|^{\frac{N-1}{2}}}\
  \Bigl(\wh{f}(\wh{x})-\wh{g}(\wh{x})\Bigr)\Bigr|^2\, dx
&=\sup_{R \ge 1} \frac1R\int_0^R \int_{S^{N-1}}|\wh{(f-g)}(\omega)|^2\, d\sigma(\omega)\, dr\\
&\leq C(p)\|f-g\|_{L^{p'}(\R^N)}^2<C(p)\eps^2. 
\end{align*}
with a constant $C(p)>0$ independent of $g$. Combining these estimates, we find
\begin{align*}
&\limsup_{R\to\infty}\frac{1}{R}\int\limits_{B_R}\Bigl|\cR
  f(x)-\kappa_N\frac{e^{i|x|}}{|x|^{\frac{N-1}{2}}}\ \wh{f}(\wh{x})\Bigr|^2\, dx 
  \leq \limsup_{R\to\infty}\frac{3}{R}\int_{B_R}\Bigl|\cR f(x)-\cR g(x)\Bigr|^2\, dx \\
&+
\limsup_{R\to\infty}\frac{3}{R}\int\limits_{B_R}\Bigl|\cR
  g(x)-\kappa_N\frac{e^{i|x|}}{|x|^{\frac{N-1}{2}}}\ \wh{g}(\wh{x})\Bigr|^2\, dx\\
&\qquad \qquad \qquad \qquad \qquad \qquad \qquad
+\limsup_{R\to\infty}\frac{3|\kappa_N|^2}{R}\int_{B_R}\Bigl|\frac{e^{i|x|}}{|x|^{\frac{N-1}{2}}}
  \Bigl( \wh{f}(\wh{x})-\wh{g}(\wh{x})\Bigr)\Bigr|^2\, dx\\
&< 3C^2 \eps^2 +3|\kappa_N|^2C(p)\eps^2.
\end{align*}
Since this holds for every $\eps>0$, \eqref{eqn:weak_asympt} follows.
\end{proof}

In the case where the function $f$ is integrable and exhibits some appropriate decay at infinity, 
a pointwise asymptotic expansion can also be obtained.

\begin{proposition}\label{prop:farfield_N}
Let $f\in L^1(\R^N)$ satisfy $|f(x)|\leq C|x|^{-N-\eps}$ for all $x\neq 0$ for some 
$C, \eps>0$. Then, 
$$
[\cR f](x)= \sqrt{\frac{\pi}{2}}\ \frac{e^{i|x|-\frac{i(N-3)\pi}{4}}}{|x|^{\frac{N-1}{2}}}\ \wh{f}\left(\wh{x}\right) 
+ o(|x|^{-\frac{N-1}{2}}), \quad\text{as $|x|\to\infty$.}
$$
\end{proposition}
\begin{proof}
In order to prove the asymptotic expansion for $\cR f=\Phi\ast f$ we split the integral
\begin{align*}
(\Phi\ast f)(x)&=C_N\int_{\R^N}\frac{H^{(1)}_{\frac{N-2}{2}}(|x-y|)}{|x-y|^{\frac{N-2}{2}}}f(y)\, dy,
\end{align*}
where $C_N=\frac{i}{4}(2\pi)^{\frac{2-N}{2}}$, into three parts and treat each of them separately.
Let us first consider for $x\in\R^N$ with $|x|\geq 2$,
$$
I_1(x)=\int\limits_{B_1(x)}\frac{H^{(1)}_{\frac{N-2}{2}}(|x-y|)}{|x-y|^{\frac{N-2}{2}}}f(y)\, dy.
$$
Using the asymptotic property \eqref{eq:3}, we can find a constant $\kappa_1>0$ such that
for all $x, y\in\R^N$ with $|x-y|<1$, $|H^{(1)}_{\frac{N-2}{2}}(|x-y|)|\leq \kappa_1 \ |x-y|^{\frac{2-N}{2}}$. Therefore,
using the decay property of $f$, we find
\begin{align*}
|I_1(x)|&\leq \kappa_1\int\limits_{B_1(x)}|x-y|^{2-N}|f(y)|\, dy\leq C\kappa_1\int\limits_{B_1(x)}|x-y|^{2-N}|y|^{-N-\eps}\, dy\\
&\leq 2^{N+\eps}C\kappa_1|x|^{-N-\eps}\omega_N\int_0^1r\, dr=2^{N-1+\eps}\omega_NC\kappa_1\ |x|^{-N-\eps}.
\end{align*}
Next, we set $A(x)=\{y\in\R^N\, :\, |x-y|>1\text{ and }|y|\geq\sqrt{|x|}\}$ and consider
$$
I_2(x)=\int\limits_{A(x)}\frac{H^{(1)}_{\frac{N-2}{2}}(|x-y|)}{|x-y|^{\frac{N-2}{2}}}f(y)\, dy.
$$
From \eqref{eq:3} we can find some constant $\kappa_2>0$ such that for every $|x|\geq 1$,
\begin{align*}
|I_2(x)|&\leq\kappa_2\int\limits_{A(x)}|x-y|^{\frac{1-N}{2}}|f(y)|\ dy\\
&\leq \kappa_2\left(\frac{|x|}{2}\right)^{\frac{1-N}{2}}\int\limits_{A(x)}\frac{|x-y|^{\frac{N-1}{2}}+|y|^\frac{N-1}{2}}{|x-y|^{\frac{N-1}{2}}}\, |f(y)|\, dy\\
&\leq \kappa_2\left(\frac{|x|}{2}\right)^{\frac{1-N}{2}}\Bigl(\int\limits_{A(x)}|f(y)|\, dy
+C\int\limits_{A(x)}|x-y|^{\frac{1-N}{2}}|y|^{-\frac{(N+1+2\eps)}{2}}\, dy\Bigr).
\end{align*}
Since $f\in L^1(\R^N)$, the first integral on the last line goes to zero uniformly as $|x|\to\infty$, and using 
\cite[Appendix 2, Lemma 1]{alsholm-schmidt70}, we obtain
$$
\int\limits_{A(x)}|x-y|^{\frac{1-N}{2}}|y|^{-\frac{(N+1+2\eps)}{2}}\, dy\to 0,\quad\text{uniformly as }|x|\to\infty.
$$
Hence, $I_2(x)=o(|x|^{-\frac{N-1}{2}})$ as $|x|\to\infty$.

In a last step, we study for $x\in\R^N$ with $|x|\geq 4$ the integral
$$
I_3(x)=C_N\int\limits_{D(x)}\frac{H^{(1)}_{\frac{N-2}{2}}(|x-y|)}{|x-y|^{\frac{N-2}{2}}}f(y)\, dy,
$$
where $D(x)=\{y\in\R^N\, :\, |x-y|>1\text{ and }|y|\leq\sqrt{|x|}\}$. 
We first notice that by \eqref{eq:3} we can write, as in the proof of Proposition \ref{prop:farfield},
\begin{align*}
I_3(x)&=\gamma_N\int\limits_{D(x)}\frac{e^{i|x-y|}}{|x-y|^{\frac{N-1}{2}}}\Bigl(1+\delta(|x-y|)\Bigr)f(y)\, dy,
\end{align*}
where $\gamma_N=\frac{(2\pi)^{\frac{1-N}{2}}}{2}e^{-\frac{i(N-3)\pi}{4}}$ and 
$\delta_\ast:=\sup\limits_{r\geq 1}r|\delta(r)|<\infty$.
As a consequence, we obtain as in Proposition \ref{prop:farfield}, using \eqref{eqn:11},
\begin{align*}
&\Bigl|\int\limits_{D(x)}\Bigl(\frac{e^{i|x-y|}}{|x-y|^{\frac{N-1}{2}}}-\frac{e^{i|x|-i\wh{x}\cdot y}}{|x|^{\frac{N-1}{2}}}\Bigr)f(y)\, dy\Bigr|
\leq \int\limits_{D(x)}\Bigl(\frac{\sqrt{2}\zeta |y|^2}{|x|\ |x-y|^{\frac{N-1}{2}}}+\frac{\zeta|y|}{|x|^{\frac{N+1}{2}}}\Bigr)\ |f(y)|\, dy\\
&\qquad\leq \frac{(2^{\frac{N-1}{2}}\sqrt{2}+1)\zeta}{|x|^{\frac{N+1}{2}}} \int\limits_{\{|y|<1\}}|f(y)|\, dy\;
+\frac{(2^{\frac{N-1}{2}}\sqrt{2}+1)\zeta}{|x|^{\frac{N-1}{2}+\frac{\eps}{4}}}\!\!\!\!\!\int\limits_{\{1\leq |y|<\sqrt{|x|}\}} |y|^{-N-\frac{\eps}{2}}\, dy\\
&\qquad\leq \kappa_3 |x|^{-\frac{N-1}{2}-\frac{\eps}{4}}
\end{align*}
for some constant $\kappa_3>0$, and, moreover,
\begin{align*}
\left|\,\int\limits_{D(x)}\frac{e^{i|x-y|}}{|x-y|^{\frac{N-1}{2}}}\delta(|x-y|)f(y)\, dy\right|
\leq \frac{2^{\frac{(N+1)}{2}}\delta_\ast}{|x|^{\frac{(N+1)}{2}}}\|f\|_{L^1(\R^N)}.
\end{align*}
Combining these last two estimates, we obtain
$$
\left|I_3(x)-\gamma_N\frac{e^{i|x|}}{|x|^{\frac{N-1}{2}}}\int_{D(x)}e^{-i\wh{x}\cdot y}f(y)\, dy\right|
\leq|\gamma_N| (\kappa_3 +2^{\frac{(N+1)}{2}}\delta_\ast\|f\|_{L^1(\R^N)})\ |x|^{-\frac{N-1}{2}-\frac{\eps}{4}},
$$
and using the fact that, 
$$
\left|\int_{\R^N\backslash D(x)}e^{-i\wh{x}\cdot y}f(y)\, dy\right| \leq \int\limits_{B_1(x)}|f(y)|\, dy + \int\limits_{\{|y|\geq\sqrt{|x|}\}}|f(y)|\, dy\to 0,
$$
uniformly as $|x|\to\infty$, we can write
$$
I_3(x)=\gamma_N\frac{e^{i|x|}}{|x|^{\frac{N-1}{2}}}(2\pi)^\frac{N}{2}\wh{f}(\wh{x}) +o(|x|^{-\frac{(N-1)}{2}}), \quad\text{as }|x|\to\infty,
$$
and the claim follows.
\end{proof}

As a consequence of Corollary~\ref{coroll:asympt_inhom_phi} and the above expansions, 
we obtain that in the case where $\frac{2(N+1)}{N-1}\leq p\leq \frac{2N}{N-2}$ and
$f\in L^{p'}(\R^N)$ is real-valued, there exists for each solution $u\in L^p(\R^N)$ of
$$
u=\mR(f)=\Re(\cR f)
$$
a unique $g\in L^2(S^{N-1})$ such that
$$
u(x)=\cR(f)(x)+\int_{S^{N-1}}e^{ix\cdot\xi}g(\xi)\, d\sigma(\xi)\quad\text{for a.e. }x\in\R^N,
$$
and a direct identification gives
$$
\int_{S^{N-1}}e^{ix\cdot\xi}g(\xi)\, d\sigma(\xi)=-i\Im(\cR f)(x)=\frac{1}{2}\Bigl(\overline{[\cR f](x)}-[\cR f](x)\Bigr).
$$
Comparing the expansions in \eqref{eqn:expansion_inhom} and \eqref{eqn:weak_asympt}, we find
\begin{equation}\label{eqn:g}
g=-\frac{i}{4} (2\pi)^{\frac{2-N}{2}} \wh{f} \in L^2(S^{N-1}).
\end{equation}

We conclude this study of the operator $\cR$ by stating a result on the asymptotic decay of solutions of convolution equations in which the kernel 
has the same asymptotics as the fundamental solution $\Phi$ above. We shall use this result in Section \ref{sec:dual} 
below in order to obtain a pointwise asymptotic expansion for real-valued solutions of the nonlinear
Helmholtz equation.
\begin{lemma}\label{lem:asympt_fct}
Let $V\in L^q(\R^N)\cap L^s(\R^N)$ with $q<\frac{2N}{N+1}\leq\frac{N}{2}<s$ and consider a measurable function $u$: $\R^N$ $\to$ $\R$ 
satisfying $Vu\in L^1(\R^N)\cap L^s(\R^N) $ and
$$
u=K\ast(Vu),
$$ 
where $|K(x)|\leq C_0\max\{|x|^{\frac{1-N}{2}},|x|^{2-N}\}$ for $x\neq 0$.
Then there exists a constant $C>0$ such that
$$
|u(x)|\leq C|x|^{\frac{1-N}{2}}\quad\text{ for all }x\neq 0.
$$
\end{lemma}
\begin{proof}
The proof is based on an iteration procedure, similar to the one used by Zemach and Odeh in \cite{zemach-odeh60} (see also \cite{alsholm-schmidt_pre}).
We start by remarking that H\"older's inequality gives for $\sigma\in\{\frac{N-1}{2},N-2\}$, $R>0$ and $x\in\R^N$, 
\begin{align*}
&\int_{M_R}|V(y)|\ |x-y|^{-\sigma}\, dy\leq\left(\int_{M_R}|V(y)|^s dy\right)^\frac{1}{s}\left(\int_{B_1(x)} |x-y|^{-s'\sigma}\, dy\right)^\frac{1}{s'}\\
&+\left(\int_{M_R}|V(y)|^q\, dy\right)^\frac{1}{q} \left(\int_{\R^N\backslash B_1(x)}|x-y|^{-q'\sigma}\, dy\right)^{\frac{1}{q'}},
\end{align*}
which tends to $0$, as $R\to\infty$, uniformly in $x$. Indeed, our assumptions ensure that $s'=\frac{s}{s-1}<\frac{N}{N-2}$ and $q'>\frac{2N}{N-1}$, and therefore
$s'\sigma<N$ whereas $q'\sigma>N$. From now on,
we choose $R>1$ such that
\begin{equation}\label{eqn:VR}
C_0\sup_{x\in\R^N 
}\int_{M_R}|V(y)|\left( |x-y|^{\frac{1-N}{2}}+|x-y|^{2-N}\right)\, dy<2^{-N}.
\end{equation}
Letting 
$$
u_0(x)=\int_{B_R}K(x-y)V(y)u(y)\, dy,
$$
we infer that for all $|x|\geq 2R$,
\begin{equation}
\label{revised-additional}
|u_0(x)|\leq C_0\int_{B_R}|x-y|^{\frac{1-N}{2}}\bigl|V(y)u(y)\bigr|\, dy\leq C_1|x|^{\frac{1-N}{2}},
\end{equation}
with $C_1=2^{\frac{N-1}{2}}C_0\|Vu\|_1$. We now set for $|x|\geq R$,
$$
B_0(x)=\int_{M_R}K(x-y)V(y)u(y)\, dy
$$
and define inductively for $k\geq 1$,
\begin{align*}
u_k(x)&=\int_{M_R}K(x-y)V(y)u_{k-1}(y)\, dx,\\
\text{ and }B_k(x)&=\int_{M_R}K(x-y)V(y)B_{k-1}(y)\, dx.
\end{align*}
Thus, for each $m\in\N$,
$$
u=\sum_{k=0}^m u_k + B_m.
$$
Setting $\beta_k=\sup\limits_{|x|\geq R}|B_k(x)|$, a similar calculation as above gives $\beta_0<\infty$, since $Vu\in L^1(\R^N)\cap L^s(\R^N)$. 
Moreover, using \eqref{eqn:VR}, we obtain $\beta_k\leq 2^{-N}\beta_{k-1}$ for all $k\geq 1$,
and thus $\beta_m\leq 2^{-mN}\beta_0\to 0$ as $m\to\infty$,
showing that 
$$
u=\sum\limits_{k=0}^\infty u_k
$$ 
holds uniformly in $M_R$. Next, we remark
that $\mu_0:=\sup\limits_{|x|\geq R}|x|^{\frac{N-1}{2}}|u_0(x)|<\infty$ by (\ref{revised-additional}) and since $u_0\in L^\infty(\R^N)$. Hence, setting
$\mu_k=\sup\limits_{|x|\geq R}|x|^{\frac{N-1}{2}}|u_k(x)|$ for $k\geq 1$, we obtain
\begin{align*}
&|x|^{\frac{N-1}{2}}|u_k(x)|\leq \mu_{k-1} C_0|x|^{\frac{N-1}{2}}\int_{M_R}|V(y)|\ |y|^\frac{1-N}{2}\left(|x-y|^{\frac{1-N}{2}}+|x-y|^{2-N}\right)\, dy \\
&\leq \mu_{k-1} 2^{\frac{N-3}{2}}C_0\int_{M_R}|V(y)| (1+|y|^{\frac{1-N}{2}}|x-y|^{\frac{N-1}{2}})\left(|x-y|^{\frac{1-N}{2}}+|x-y|^{2-N}\right)\, dy\\
&\leq \frac12\mu_{k-1},
\end{align*}
for all $|x|\geq R$, where 
Young's inequality 
and \eqref{eqn:VR} have been used. Iterating the preceding estimate, we obtain $\mu_k\leq 2^{-k}\mu_0$
and consequently,
$$
\sup_{|x|\geq R}|x|^{\frac{N-1}{2}}|u(x)|\leq \mu_0 \sum_{k=0}^\infty 2^{-k}=2\mu_0<\infty,
$$
which concludes the proof.
\end{proof}

\section{The nonvanishing property}
\label{sec:nonv-prop}

As in the preceding section, all function spaces are understood as
spaces of complex-valued functions. The following theorem is a key 
ingredient to obtain real standing wave solutions of
(\ref{eqn:1}) via variational methods in the case where $Q(x)$ does not
vanish asymptotically as $|x| \to \infty$. We believe that it might also have
further applications to the study of complex solutions of (\ref{eqn:1}).

\begin{theorem}{ (nonvanishing property)}\\
\label{lem:conc_comp}
 Let $N\geq 3$ and $\frac{2(N+1)}{N-1}<
 p<\frac{2N}{N-2}$. Moreover, let $(v_n)_n\subset L^{p'}(\R^N)$ be a bounded sequence satisfying 
$\limsup \limits_{n\to\infty}\left|\int_{\R^N}v_n\cR v_n\,
  dx\right|>0.$ Then there exists $R>0$, $\zeta>0$ and a sequence
$(x_n)_n\subset\R^N$ such that, up to a subsequence, 
\begin{equation}\label{eqn:liminf2}
\int_{B_R(x_n)}|v_n|^{p'}\, dx \geq \zeta\quad\text{for all }n.
\end{equation}
\end{theorem}

The remainder of this section is devoted to the proof of this
result. We fix $\psi \in \cS$ such that $\wh{\psi}\in\cC^\infty_c(\R^N)$ is radial, $0\leq \wh{\psi}\leq 1$,
$\wh{\psi}(\xi)=1$ for $| |\xi|-1|\leq\frac16$ and $\wh{\psi}(\xi)=0$
for $| |\xi|-1|\geq\frac14$. We then write $\Phi= \Phi_1 + \Phi_2$
with 
$$
\Phi_1:= \psi * \Phi, \qquad \Phi_2 = \Phi-\Phi_1.
$$
From (\ref{eq:4}) it then follows, by making $C_0>0$ larger if
necessary, that $\Phi_1 \in \cC^\infty(\R^N)$ and  
\begin{equation}
\label{eq:15}
|\Phi_1(x)| \le C_0 (1+|x|)^{\frac{1-N}{2}} \qquad
  \text{for $x \in \R^N$.}   
\end{equation}
This in particular implies that $|\Phi_2(x)|= |[\Phi-\Phi_1](x)| \le 2
C_0 |x|^{2-N}$ for $|x| \le 1$.  Moreover, since $\wh{\Phi_2}=
(1-\wh{\psi}) \wh{\Phi}$, and $\wh{\Phi} (\xi)= (|\xi|^2-1-i0)^{-1}$
with the notation of
\cite{gelfand}, we have $\wh{\Phi_2} \in
\cC^\infty(\R^N)$ with $\wh{\Phi_2}(\xi)= (|\xi|^2-1)^{-1}$ for
$|\xi| \ge \frac{5}{4}$. This implies that 
$\partial^\gamma \wh{\Phi_2}\in L^1(\R^N)$ for all $\gamma\in\N_0^N$ such that $|\gamma|>N-2$, which gives 
$|\Phi_2(x)|\leq \kappa_s|x|^{-s}$, $x\in\R^N$ for all $s>N-2$ with some
constant $\kappa_s>0$. In particular, by making $C_0>0$ larger if
necessary, we have 
\begin{equation}
  \label{eq:17}
|\Phi_2(x)|\leq C_0 \min\{|x|^{2-N},|x|^{-N}\} \qquad \text{for $x
  \in \R^N \setminus \{0\}.$}
\end{equation}

We first prove a variant of Theorem~\ref{lem:conc_comp} related to
$\Phi_2$.

\begin{lemma}
\label{sec:resolv-estim-cocomp-3}
Let $2< p < \frac{2N}{N-2}$, and suppose that $(v_n)_n
\subset \cS$ is a bounded sequence in
$L^{p'}(\R^N)$ such that  
\begin{equation}
\label{eq:11}
\lim_{n\to\infty}\left(\sup_{y\in\R^N}\int_{B_\rho(y)}|v_n|^{p'}\,
  dx\right)=0 \qquad \text{for all $\rho>0$.}
\end{equation}
Then 
\begin{equation}
  \label{eq:10}
\int_{\R^N} v_n [\Phi_2 * v_n]\,dx \to 0 \qquad \text{as $n \to \infty$.}  
\end{equation}
\end{lemma}

\begin{proof}
Setting $A_R:= \{x \in
\R^N\::\: \frac{1}{R} \le |x| \le R\}$  and $D_R:= \R^N \setminus
A_R$ for $R >1$, we derive from (\ref{eq:17}) that  
\begin{equation}
  \label{eq:13}
\|\Phi_2\|_{L^{\frac{p}{2}}(D_R)} \to 0 \qquad \text{as $R \to \infty$},  
\end{equation}
since $1<\frac{p}{2}< \frac{N}{N-2}$. Hence, by Young's inequality,
\begin{equation}
  \label{eq:14}
\sup_{n \in \N} \Bigl|\int_{\R^N} v_n [(1_{D_R} \Phi_2)*v_n]\,dx\Bigr| \le
\|\Phi_2\|_{L^{\frac{p}{2}}(D_R)} \sup_{n \in \N} \|v_n\|_{L^{p'}(\R^N)}^2 \to 0
\qquad \text{as $R \to \infty$.} 
\end{equation}
Moreover, decomposing $\R^N$ into disjoint $N$-cubes $\{Q_\ell\}_{\ell\in\N}$ 
of side length $R$, and considering for each $\ell$ the $N$-cube $Q'_\ell$ 
with the same center as $Q_\ell$ but with side length $3R$, we find,
similarly as in \cite[pp. 109-110]{alama-li92},
\begin{align*}
\Bigl|&\int_{\R^N} v_n [(1_{A_R} \Phi_2)*v_n]\,dx \Bigr|
 \leq \sum_{\ell=1}^\infty\int_{Q_\ell}\Bigl(\int_{\frac{1}{R}<|x-y|<R}|\Phi_2(x-y)\,||v_n(x)|\, |v_n(y)|\, dy\Bigr) dx\\
 &\leq
 C R^{N-2}\sum_{\ell=1}^\infty\int_{Q_\ell}\Bigl(\int_{Q'_\ell}|v_n(x)|\,  |v_n(y)|\, dy\Bigr) dx\\ 
 &\leq
 C R^{N-2+ \frac{2N}{p}}\sum_{\ell=1}^\infty\Bigl(\int_{Q'_\ell}|v_n(x)|^{p'}\,dx\Bigr)^{\frac{2}{p'}}
\\
 & \leq CR^{N-2+ \frac{2N}{p}}
 \Bigl[\:\sup_{\ell \in \N}\int_{Q'_\ell}|v_n(x)|^{p'}\,
   dx\Bigr]^{\frac{2}{p'}-1} \:\sum_{\ell=1}^\infty
   \int_{Q'_\ell}|v_n(x)|^{p'}\,dx\\ 
&\leq CR^{N-2+ \frac{2N}{p}} \Bigl[\:\sup_{y \in \R^N}\int_{B_{3R\sqrt{N}}(y)}|v_n(x)|^{p'}\,
   dx\Bigr]^{\frac{2}{p'}-1}\: 3^N \|v_n\|_{p'}^{p'} , 
\end{align*}
so by assumption (\ref{eq:11}) we have 
\begin{equation}
  \label{eq:12}
\qquad \lim_{n \to \infty} \int_{\R^N} v_n [(1_{A_R} \Phi_2)*v_n]\,dx
= 0 \qquad \text{for every $R>0$.}
\end{equation}
Combining (\ref{eq:14}) and (\ref{eq:12}), we thus obtain
(\ref{eq:10}), as claimed.
\end{proof}

We also need the following Lemma which is related to 
\cite[Lemma 1]{gutierrez04} (and the remarks before and after that lemma).

\begin{proposition}
\label{sec:resolv-estim-cocomp-1}
Let $p > \frac{2(N+1)}{N-1}$,
$\lambda_{p}:=\frac{N-1}{2} -\frac{N+1}{p}>0$ and $M_R:= \R^N \setminus
 B_R$ for $R>0$. Then there exists a
constant $C>0$ such that, for $R \ge 1$,  
$$
\|[1_{M_R} \Phi_1 ] \ast f\|_p\leq C R^{-\lambda_p}
\|f\|_{p'} \qquad \text{for $f\in\cS$ with $\supp \wh{f} \subset \{\xi \::\:
||\xi|-1|\le \frac{1}{2}\}$.} 
$$
\end{proposition}

\begin{proof}
It suffices to prove the assertion for $R \ge 4$. Put $P_R:= 1_{M_R} \Phi_1 $ for $R \ge 4$, 
and fix a radial, nonnegative function $\eta \in \cC_c^\infty(\R^N)$ such that
$\eta(x)= 1$ if $0 \le |x| \le 1$, $\eta(x)= 0$ if $|x|\ge
2$. Moreover, for $j \in \N$, define $\phi_j \in \cC_c^\infty(\R^N)$
by $\phi_j(x)= \eta(x/2^j)- \eta(x/2^{j-1})$. Since $P_R \equiv 0$ on $B_R$, we then have the corresponding dyadic
decomposition 
\begin{equation}
  \label{eq:8}
P_R= \sum \limits_{j=[\log_2 R] }^\infty P^j \qquad \text{with $P^j(x):= P_R(x)
\phi_j(x)$ for $j \in \N$,} 
\end{equation}
so that 
\begin{equation}
  \label{eq:16}
\|P^j\|_\infty \le C_1 2^{-\frac{j(N-1)}{2}} \qquad \text{for all $j
  \ge [\log_2 R]$}  
\end{equation}
with $C_1:= 2^{\frac{N-1}{2}}C_0$ and $C_0>0$ as in (\ref{eq:15}). 
In the following, the letter $C$ stands for (possibly different)
positive constants independent of $R$. We claim that 
\begin{equation}
  \label{eq:5}
\|P^j \ast f\|_2 \leq C \, 2^{\frac{j}{2}}
\|f\|_{\frac{2(N+1)}{N+3}}
\end{equation}
for $j\geq [\log_2 R]$ and $f\in\cS$ with $\supp \,\wh{f} \subset \{\xi \::\:
||\xi|-1|\le \frac{3}{4}\}$. This follows almost exactly as in the proof of \cite[Lemma
1]{gutierrez04}, but we repeat the argument for the convenience of the
reader. Since $\wh{P^j}$ is a radial function, we also write $\wh{P^j}(r)$ in place of
$\wh{P^j}(\xi)$ if $|\xi|=r$ in the following. Setting
$q=\frac{2(N+1)}{N+3}$, we then have, using Theorem~\ref{sec:resolv-estim-nonv},  
\begin{align*}
\|P^j \ast f\|_2^2&=C\!\!\int\limits_{||\xi|-1|\leq\frac34}\!\!\!|\wh{P^j}(\xi)\wh{f}(\xi)|^2\,
d\xi = C\int_{\frac14}^{\frac74}r^{N-1}|\wh{P^j}(r)|^2\int_{S^{N-1}}|\wh{f}(r\omega)|^2\, d\sigma(\omega) dr\\
&\leq C \|f\|_{q}^2 \int_{\frac14}^{\frac74}|\wh{P^j}(r)|^2\, dr
 \leq C \|f\|_{q}^2\int_{\R^N}|P^j(x)|^2\, dx \le C 2^j \|f\|_{q}^2.
\end{align*}
Hence (\ref{eq:5}) holds. We now fix $\phi \in \cS$ such that $\wh{\phi} \in \cC^\infty_c(\R^N)$ 
is radial, nonnegative and such that
$\wh{\phi}(\xi) \equiv 1$ on $\{\xi \::\:
||\xi|-1|\le \frac{1}{2}\}$ and $\wh{\phi}(\xi) \equiv 0$ on $\{\xi \::\:
||\xi|-1|\ge \frac{3}{4}\}$. We also define $Q^j:= P^j \ast
\phi$. As a consequence of (\ref{eq:5}), we then have 
\begin{equation}
\label{eq:6}
\|Q^j \ast f\|_2 = \|P^j \ast (\phi \ast f)\|_2 \le C \, 2^{\frac{j}{2}}
\|\phi \ast f\|_{\frac{2(N+1)}{N+3}} 
\leq C 2^{\frac{j}{2}}
\|f\|_{\frac{2(N+1)}{N+3}}
\end{equation}
for all $f\in\cS$ and $j\geq [\log_2 R]$, since $\supp \,\wh{\phi \ast f} = \supp\, \wh{\phi}\wh{f}
\subset \{\xi \::\:
||\xi|-1|\le \frac{3}{4}\}$. Note that in the last step we used
Young's inequality (replacing $C \|\phi\|_1$ by $C$). By duality, this implies that 
\begin{equation}
\label{eq:6bis}
\|Q^j \ast f\|_{\frac{2(N+1)}{N-1}} \leq C \, 2^{\frac{j}{2}} \|f\|_{2} \qquad \text{for all $f\in\cS$}
\end{equation}
and $j\geq [\log_2 R]$. Setting $s=\frac{2(N+1)}{N}$, so that $\frac{1}{s}=\frac{1}{2}\Bigl(\frac{1}{2} + \frac{N-1}{2(N+1)}\Bigr)$,
we therefore obtain by complex interpolation that 
\begin{equation}
\label{eq:6-1}
\|Q^j \ast f\|_{s} \leq C \, 2^{\frac{j}{2}} \|f\|_{s'} \qquad \text{for all $f\in\cS$}
\end{equation}
and $j\geq [\log_2 R]$. By (\ref{eq:16}) we also have 
$$
\|Q^j\|_\infty \le \|P^j\|_\infty \|\phi\|_1 \le C_1 \|\phi\|_1 2^{-\frac{j(N-1)}{2}} \qquad \text{for $j\geq [\log_2 R]$,}
$$
so that, by Young's inequality,  
\begin{equation}
  \label{eq:7}
\|Q^j \ast f\|_\infty \leq C 2^{-\frac{j(N-1)}{2}}  \|f\|_{1} \qquad \text{for all $f\in\cS$}
\end{equation}
and $j\geq [\log_2 R]$. Combining (\ref{eq:6-1}) and (\ref{eq:7}) and applying complex
interpolation again, we find that 
$$
\|Q^j \ast f\|_{p} \leq C \,
2^{j\left(\frac{(N+1)}{p}-\frac{(N-1)}{2}\right)}\|f\|_{p'} = C \,
2^{-j\lambda_p}\|f\|_{p'}
$$
for all $f\in\cS$, $p \ge s$ and $j\geq [\log_2 R]$. As in the
assumption of the proposition, we now restrict our attention to
$p>\frac{2(N+1)}{N-1} > s$, so that $\lambda_p>0$. Recalling (\ref{eq:8}) and 
using that $Q^j \ast f = P^j \ast f$ for all $f\in\cS$ with $\supp \wh{f} \subset \{\xi \::\:
\bigl||\xi|-1\bigr|\le \frac{1}{2}\}$, we then conclude that 
$$
\|[1_{M_R} \Phi_1 ] \ast f\|_p = \|P_R \ast f\|_p \le C \|f\|_{p'} \sum_{j=[\log_2
  R]}^\infty 2^{-j\lambda_p} \le  C R^{-\lambda_p} \|f\|_{p'}
$$
for all $f\in\cS$ with $\supp \wh{f} \subset \{\xi \::\:
\bigl||\xi|-1\bigr|\le \frac{1}{2}\}$, as claimed.
\end{proof}

\begin{lemma}
\label{sec:resolv-estim-cocomp-4}
Let $p > \frac{2(N+1)}{N-1}$, and suppose that $(v_n)_n
\subset \cS$ is a bounded sequence in
$L^{p'}(\R^N)$ such that  
\begin{equation*}
\lim_{n\to\infty}\left(\sup_{y\in\R^N}\int_{B_\rho(y)}|v_n|^{p'}\,
  dx\right)=0 \qquad \text{for all $\rho>0$.}
\end{equation*}
Then 
\begin{equation}
  \label{eq:10-2}
\int_{\R^N} v_n [\Phi_1 * v_n]\,dx \to 0 \qquad \text{as $n \to \infty$.}  
\end{equation}
\end{lemma}

\begin{proof}
Fix a radial function $\phi \in \cS$ such that
$\wh{\phi}\in\cC^\infty_c(\R^N)$ is radial, $0\leq \wh{\phi}\leq 1$,
$\wh{\phi}(\xi)=1$ for $| |\xi|-1|\leq \frac14$ and $\wh{\phi}(\xi)=0$
for $| |\xi|-1|\geq\frac12$. Moreover, let $w_n:= \phi \ast v_n \in
\cS$. 
We then have $\Phi_1 * v_n = \Phi_1 * w_n$, since
$\wh{\Phi_1}\wh{\phi}= \wh{\Phi_1}$ by construction. 
Hence 
\begin{align}
 \int_{\R^N} v_n &[\Phi_1 * v_n]\,dx = \int_{\R^N} v_n [\Phi_1 *
 w_n]\,dx  \nonumber\\ 
&= \int_{\R^N} v_n [(1_{B_R} \Phi_1 ) * w_n]\,dx +  \int_{\R^N} v_n
[(1_{M_R} \Phi_1 ) * w_n]\,dx  \quad \text{for every $n$.}\label{eq:18}
\end{align}
Since $\supp \,\wh{w_n} \subset \{\xi \::\:
||\xi|-1|\le \frac{1}{2}\}$ for every $n \in \N$,
Proposition~\ref{sec:resolv-estim-cocomp-1} implies that 
$$
\Bigl|\int_{\R^N} v_n [(1_{M_R} \Phi_1)  * w_n]\,dx\Bigr| \le C\|v_n\|_{p'}
\|w_n\|_{p'} R^{-\lambda_p} \le C\|\phi\|_1
\|v_n\|_{p'}^2 R^{-\lambda_p}
$$
for every $n \in \N$, $R>4$, where we used Young's inequality in the
last step. As a consequence
\begin{equation}
  \label{eq:20}
\sup_{n \in \N} \Bigl|\int_{\R^N} v_n [(1_{M_R} \Phi_1)  * w_n]\,dx\Bigr| \to 0
\qquad \text{as $R \to \infty$.}  
\end{equation}
Moreover, decomposing $\R^N$ into disjoint $N$-cubes $\{Q_\ell\}_{\ell\in\N}$ of side length $R$, 
and considering for each $\ell$ the $N$-cube $Q'_\ell$ 
with the same center as $Q_\ell$ but with side length $3R$, we find,
arguing slightly differently than in the proof of Lemma~\ref{sec:resolv-estim-cocomp-3}, 
\begin{align*}
\Bigl|&\int_{\R^N} v_n [(1_{B_R} \Phi_1 )*w_n]\,dx \Bigr|
 \leq \sum_{\ell=1}^\infty\int_{Q_\ell}\Bigl(\int_{|x-y|<R}|\Phi_1(x-y)|\,|v_n(x)|\, |w_n(y)|\, dy\Bigr) dx\\
 &\leq
 \|\Phi_1\|_\infty 
 \sum_{\ell=1}^\infty\int_{Q'_\ell}|v_n(x)|\,dx \int_{Q'_\ell}\,
   |w_n(x)|\, dx \\
&\leq
 \|\Phi_1\|_\infty 
\Bigl[
\sum_{\ell=1}^\infty\Bigl(\int_{Q'_\ell}|w_n(x)|\,dx\Bigr)^{p'}\Bigr]^{\frac{1}{p'}} \Bigl[
\sum_{\ell=1}^\infty\Bigl(\int_{Q'_\ell}\,
   |v_n(x)|\, dx\Bigr)^{p}\Bigr]^{\frac{1}{p}} \\
&\leq C 
 R^{\frac{2N}{p}}\Bigl[
\sum_{\ell=1}^\infty \int_{Q'_\ell}|w_n(x)|^{p'}\,dx\Bigr]^{\frac{1}{p'}} \Bigl[
\sum_{\ell=1}^\infty\Bigl(\int_{Q'_\ell}\,
   |v_n(x)|^{p'}\, dx\Bigr)^{\frac{p}{p'}}\Bigr]^{\frac{1}{p}} \\
&\leq C  
 R^{\frac{2N}{p}}3^{\frac{N}{p'}}\|w_n\|_{p'} \Bigl[\:\sup_{\ell \in \N}
\int_{Q'_\ell}\,
   |v_n(x)|^{p'}\, dx\Bigr]^{\frac1{p'}-\frac1p} \Bigl[\:\sum_{\ell=1}^\infty \int_{Q'_\ell}\,
   |v_n(x)|^{p'}\, dx\Bigr]^{\frac{1}{p}} \\
&\leq C  
 R^{\frac{2N}{p}} 3^N \|w_n\|_{p'}   \Bigl[\:\sup_{y \in \R^N}
\int_{B_{3R\sqrt{N}}(y)}\,
   |v_n(x)|^{p'}\,
   dx\Bigr]^{\frac2{p'}-1}\:\|v_n\|_{p'}^{\frac{p'}{p}}\\
&\leq C  
 R^{\frac{2N}{p}} \|v_n\|_{p'}^{p'}\Bigl[\:\sup_{y \in \R^N}
\int_{B_{3R\sqrt{N}}(y)}\,
   |v_n(x)|^{p'}\, dx\Bigr]^{\frac2{p'}-1},
\end{align*}
so by assumption we have
\begin{equation}
  \label{eq:12-2}
\lim_{n \to \infty} \int_{\R^N} v_n [(1_{B_R} \Phi_1)*w_n]\,dx
= 0 \qquad \text{for every $R>0$.}
\end{equation}
Combining (\ref{eq:18}), (\ref{eq:20}) and (\ref{eq:12-2}), we obtain
(\ref{eq:10-2}), as claimed.
\end{proof}

\begin{proof}[Proof of Theorem~\ref{lem:conc_comp} (completed)]
Without loss of generality, we may assume that 
\begin{equation}
\label{sec:resolv-estim-cocomp}  
v_n \in \cS \qquad \text{for all
$n \in \N$.}
\end{equation}
Indeed, in any case we may replace $v_n$ by $\tilde v_n \in \cS$ with 
$\|v_n-\tilde v_n\|_{L^{p'}(\R^N)} \le \frac{1}{n}$ for every $n \in
\N$. Then 
\begin{align*}
\Bigl|\int_{\R^N}[v_n\cR v_n -\tilde v_n \cR \tilde v_n]\,&
  dx\Bigr | = \Bigl|\int_{\R^N} (v_n-\tilde v_n) \cR (v_n+\tilde
v_n)\,dx \Bigr|\\
&\le C \|v_n-\tilde v_n\|_{L^{p'}(\R^N)} \|v_n+\tilde
v_n\|_{L^{p'}(\R^N)} \to 0 \qquad \text{as $n \to \infty$}
\end{align*}
with $C$ as in (\ref{eq:1}) and thus $\liminf
\limits_{n\to\infty}\left|\int_{\R^N}\tilde v_n\cR \tilde v_n\,
  dx\right|= \liminf \limits_{n\to\infty}\left|\int_{\R^N}v_n\cR v_n\,
  dx\right|>0.$ Moreover, (\ref{eqn:liminf2}) holds if and only if it
holds for $\tilde v_n$ in place of $v_n$. Hence we may assume
\eqref{sec:resolv-estim-cocomp}. Suppose by contradiction that 
\eqref{eqn:liminf2} does not hold. Then 
\begin{equation}\label{eqn:vanishing}
\lim_{n\to\infty}\left(\sup_{y\in\R^N}\int_{B_\rho(y)}|v_n|^{p'}\,
  dx\right)=0 \qquad \text{for all $\rho>0$.}
\end{equation}
By Lemmas~\ref{sec:resolv-estim-cocomp-3} and
\ref{sec:resolv-estim-cocomp-4}, we thus find that 
$$
\int_{\R^N} v_n \cR v_n\,dx =  \int_{\R^N} v_n [\Phi_1 * v_n]\,dx +
\int_{\R^N} v_n [\Phi_2 * v_n]\,dx \to 0 \qquad \text{as $n \to
  \infty$,}
$$
contradicting the assumption. Hence (\ref{eqn:liminf2}) is true for
some $\zeta
, R>0$, as claimed. 
\end{proof}

\section{A dual variational framework for the nonlinear Helmholtz equation}\label{sec:dual}
Throughout this section, we assume that $N\geq 3$,
$\frac{2(N+1)}{N-1}\leq p\leq 2^\ast:=\frac{2N}{N-2}$. We write $\Psi:= \Re\,
\Phi$ for the real part of the fundamental solution $\Phi$ given in
(\ref{eqn:fund_sol}). Note that, by Theorem~\ref{thm:KRS},
\begin{equation}
  \label{eq:9}
\text{the linear operator $\mR: L^{p'}(\R^N) \to L^p(\R^N)$, $\mR(v):= \Psi \ast
  v$ is bounded.}  
\end{equation}
Here and in the following, in contrast to the previous sections, all function spaces are assumed to
consist of real-valued functions. For $Q \in L^\infty(\R^N)$
nonnegative with $Q\not\equiv 0$, we wish to set up a dual variational
framework to study solutions $u \in L^p(\R^N)$ of (\ref{eq:28}). Note
that equation (\ref{eq:28}) corresponds
to the special case $k=1$ in (\ref{eqn:1}), i.e., to 
\begin{equation}\label{eq:33}
- \Delta u - u = Q(x)|u|^{p-2}u, \qquad x\in \R^N.  
\end{equation}
 Setting $v= Q^{\frac{1}{p'}}|u|^{p-2}u$, we are thus led to consider the equation
\begin{equation}\label{eqn:integ2}
|v|^{p'-2}v=Q^{\frac1p}[\Psi * (Q^{\frac1p}v)]\quad\text{in }\R^N
\end{equation}
Before setting up the variational framework for equation
\eqref{eqn:integ2}, we study the so-called Birman-Schwinger operator (see \cite{alama-li92}) which appears on the 
right-hand side of \eqref{eqn:integ2}.
\begin{lemma}\label{lem:compact}
 Let $\frac{2(N+1)}{N-1}\leq p\leq\frac{2N}{N-2}$ and consider $Q\in L^\infty(\R^N)$, satisfying $Q(x)\geq 0$ for a.e. $x\in\R^N$.
 Then the Birman-Schwinger operator 
$$
\mK_p: L^{p'}(\R^N) \to L^p(\R^N), \qquad  \mK_p(v):=Q^{\frac1p}
\mR(Q^{\frac1p}v)
$$
is symmetric in the sense that $\int_{\R^N}w\mK_p(v)\, dx=\int_{\R^N}v\mK_p(w)\, dx$ 
 for all $v,w\in L^{p'}(\R^N)$. Moreover, if $p<\frac{2N}{N-2}$ then,
 \begin{itemize}
  \item[(i)] for any bounded and measurable set $B\subset\R^N$, the operator $1_B\mK_p$ is compact. 
  Here $1_B$ denotes the characteristic function of the set $B$.
  \item[(ii)] If, in addition, $\operatorname*{ess\ sup}\limits_{|x|\geq R}Q(x)\to 0$ as $R\to\infty$,
  then $\mK_p$ itself is compact.
 \end{itemize}
\end{lemma}
\begin{proof}
Note that $\mK_p$ is a bounded linear operator due to (\ref{eq:9}) and since $Q
\in L^\infty(\R^N)$ by assumption. We start by proving the compactness of $1_B\mK_p$ for a bounded set $B\subset\R^N$. 
For this, it is enough to show that $\mK_p(v_n)\to 0$ in $L^p(B)$ for
every sequence $(v_n)_n\subset L^{p'}(\R^N)$ such that $v_n\weakto 0$.

Let therefore $(v_n)_n\subset L^{p'}(\R^N)$ converge weakly to $0$. The boundedness of $\mR$ implies that 
$\mR(Q^{\frac1p}v_n)\weakto 0$ in $L^p(\R^N)$ and, according to Proposition \ref{prop:appendix}, there holds 
$\mR(Q^{\frac1p}v_n)\in W^{2,p'}_{\text{loc}}(\R^N)$ and for every $R>0$ 
there is some constant $\tilde{C}>0$ such that
$$
\|\mR(Q^{\frac1p}v_n)\|_{W^{2,p'}(B_R)}\leq \tilde{C}\left(\|\mR(Q^{\frac1p}v_n)\|_{L^p(\R^N)}+\|Q^{\frac1p}v_n\|_{L^{p'}(\R^N)}\right)
$$ 
holds for all $n$ (here $B_R\subset\R^N$ denotes the ball of radius $R$ centered at the origin). Consequently, $(\mR(Q^{\frac1p}v_n))_n$ 
is a bounded sequence in $W^{2,p'}(B_R)$, and from the compact embedding 
$W^{2,p'}(B_R)\hookrightarrow L^p(B_R)$, (recall that $p<\frac{2N}{N-2}$) we obtain 
$\mK_p(v_n)=Q^{\frac1p}\mR(Q^{\frac1p}v_n)\to 0$ in $L^p(B_R)$ as $n\to\infty$, using the weak convergence and the fact that 
$Q\in L^\infty(\R^N)$. The claim follows by choosing $R>0$ large enough for $B\subset B_R$ to hold.

Let us now assume that $Q_R:=\operatorname*{ess\, sup}\limits_{|x|\geq R}Q(x)\to 0$ as $R\to\infty$. 
In order to prove the compactness of $\mK_p$ in this case, we now
show that $\mK_p(v_n)\to 0$ in $L^p(\R^N)$ for every sequence $(v_n)_n\subset L^{p'}(\R^N)$ such that $v_n\weakto 0$.
Taking such a sequence $(v_n)_n$, we first note that for every $R>0$, $\|1_{B_R}\mK_p(v_n)\|_p\to 0$ as $n\to\infty$, since 
$1_{B_R}\mK_p$ is compact. Moreover, for all $n\in\N$, $R>0$.
\begin{align*}
\int_{\R^N}|(1-1_{B_R})\mK_p(v_n)|^p\, dx &\leq \|\mR(Q^{\frac1p}v_n)\|_p^p \:\operatorname*{ess\, sup}_{|x|\geq R}Q(x)\leq C Q_R,
\end{align*}
since $\bigl(\mR(Q^{\frac1p}v_n)\bigr)_n$ is bounded in $L^p(\R^N)$. As a consequence, 
$$
\limsup_{n\to\infty}\|\mK_p(v_n)\|_p\leq\limsup_{n\to\infty}\|1_{B_R}\mK_p(v_n)\|_p+\limsup_{n\to\infty}\|(1-1_{B_R})\mK_p(v_n)\|_p\leq C Q_R
$$
for all $R>0$. Letting $R\to\infty$, we obtain $\|\mK_p(v_n)\|_p\to 0$, as $n\to\infty$, and the compactness of $\mK_p$ follows.

To show that $\mK_p$ is symmetric, we first consider the case where the functions $f:=Q^\frac1pv$, $g:=Q^\frac1pw$ both belong to $\cS$.
Using the definition of $\mK_p$ and the properties of the convolution,
we then obtain
$$
\int_{\R^N}w\mK_p (v)\, dx = \int_{\R^N}g (\Psi\ast f)\, dx =
\int_{\R^N}f(\Psi\ast g)\, dx =\int_{\R^N}v\,\mK_p(w)\, dx.
$$
The conclusion then follows by a density argument.
\end{proof}

Consider now the energy functional
\begin{equation}\label{eqn:energy_funct}
\begin{aligned}
J(v)&=\frac{1}{p'}\int_{\R^N}|v|^{p'}\, dx - \frac12 \int_{\R^N}Q(x)^\frac1pv(x) \mR(Q^\frac1p v)(x)\, dx\\
    &=\frac1{p'}\|v\|_{p'}^{p'}-\frac12 \int_{\R^N}v\mK_p(v)\, dx
\end{aligned}
\end{equation}
for $v\in L^{p'}(\R^N)$. From the preceding lemma, we deduce that 
$J\in \cC^1(L^{p'}(\R^N),\R)$ with
\begin{equation}\label{eqn:gradient}
J'(v)w=\int_{\R^N}\Bigl(|v|^{p'-2}v-\mK_p(v)\Bigr)w\, dx\quad\text{for all }v,w\in L^{p'}(\R^N).
\end{equation}
Moreover, the functional $J$ has the so-called mountain pass geometry.
\begin{lemma}\label{lem:MP_geom2}
Let $\frac{2(N+1)}{N-1}\leq p\leq\frac{2N}{N-2}$ and consider $Q\in L^\infty(\R^N)$, $Q\not\equiv 0$, such that $Q(x)\geq 0$ for a.e. $x\in\R^N$.
\begin{itemize}
	\item[(i)] There exists $\delta>0$ and $0<\rho<1$ such that $J(v)\geq\delta>0$ for all $v\in L^{p'}(\R^N)$ with $\|v\|_{p'}=\rho$.
	\item[(ii)] There is $v_0\in L^{p'}(\R^N)$ such that $\|v_0\|_{p'}>1$ and $J(v_0)<0$.
	\item[(iii)] Every Palais-Smale sequence for $J$ is bounded in $L^{p'}(\R^N)$.
\end{itemize}
\end{lemma}
\begin{proof}
(i) As a consequence of (\ref{eq:9}) and the assumption $Q \in L^\infty(\R^N)$, there exists some constant $C>0$ such that $\|\mK_p(v)\|_p\leq C\|v\|_{p'}$ 
for all $v\in L^{p'}(\R^N)$.
Hence, if $\|v\|_{p'}=\rho$, we obtain
 $$
 J(v)= \frac{1}{p'}\rho^{p'}-\frac12\int_{\R^N}v\mK_p(v)\, dx 
 \geq \frac{1}{p'}\rho^{p'}-\frac{\rho}{2}\|\mK_p(v)\|_p\geq \frac{1}{p'}\rho^{p'}-\frac{C}{2}\rho^2>0
 $$
 for $\rho>0$ small enough, since $p'<2$.\\
(ii) From (\ref{eq:3}) it follows that there exists $r>0$ such that $\Psi(x)>0$  for all $x\in B_{2r}(0)$.
Moreover, since $Q\geq 0$ a.e. on $\R^N$ and $Q\not\equiv 0$, the metric density of the set $\omega_Q:=\{x\in\R^N\, :\, Q(x)>0\}$ (see \cite[\S 7.12]{rudin_ra}) is $1$
for almost every point from this set. Consequently, there exists $x_0\in\R^N$ and $0<\rho<r$ such that $\omega_Q\cap B_{\frac{\rho}{2}}(x_0)$ has positive measure.
Choosing $z\in \cC_c^\infty(\R^N)$ with $\text{supp }z\subset
B_\rho(x_0)$, $0\leq z\leq 1$ in $\R^N$ and $z=1$ in
$B_{\frac{\rho}{2}}(x_0)$, the definition of $\mK_p$ then implies
\begin{align*}
\int_{\R^N}z\mK_pz\, dx&=\int_{\R^N}\int_{\R^N}Q(x)^\frac1p z(x) \Psi(x-y)Q(y)^\frac1p z(y)\, dy dx\\
&\geq \int_{B_{\frac{\rho}{2}}(x_0)}\int_{B_{\frac{\rho}{2}}(x_0)}\Psi(x-y)Q(x)^\frac1pQ(y)^\frac1p\, dxdy>0.
\end{align*}
For $t>0$ we obtain
\begin{align*}
J(tz)&=\frac{t^{p'}}{p'}\int_{\R^N}|z|^{p'}\, dx - \frac{t^2}{2}\int_{\R^N}z\mK_p z\, dx\\
&=t^2\left(\frac{1}{p't^{2-p'}}\int_{\R^N}|z|^{p'}\, dx - \frac12\int_{\R^N}z\mK_p z\, dx\right)<0,
\end{align*}
provided $t$ is large enough.\\
(iii) Let $(v_n)_n\subset L^{p'}(\R^N)$ be a Palais-Smale sequence, i.e., there holds $\sup\limits_n|J(v_n)|<\infty$ and
$J'(v_n)\to 0$ in $L^{p'}(\R^N)^\ast\cong L^p(\R^N)$ as $n\to\infty$. Since
$$
J(v_n)=(\frac{1}{p'}-\frac12)\|v_n\|_{p'}^{p'}+\frac12J'(v_n)v_n
\geq(\frac{1}{p'}-\frac12)\|v_n\|_{p'}^{p'}-\frac12\|J'(v_n)\|_\ast\|v_n\|_{p'}
$$
and $1<p'<2$ holds, we infer that $(v_n)_n$ is bounded in $L^{p'}(\R^N)$.
\end{proof}

We now show that every critical point of $J$ is indeed a solution of our original problem.
For this, we first note that for $v\in L^{p'}(\R^N)$, $J'(v)=0$ if and only if it satisfies \eqref{eqn:integ2}.
Setting 
\begin{equation}\label{eqn:u}
u=\mR(Q^\frac1p v),
\end{equation} 
we find that $u$ solves the equation
\begin{equation}\label{eqn:integ12}
u= \mR(Q|u|^{p-2}u).
\end{equation}
Remark furthermore, that $u\not\equiv 0$ if $v\not\equiv 0$, since the condition $J'(v)=0$ implies
$\|v\|_{p'}^{p'}=\int_{\R^N}Q^\frac1pv u\, dx$.
In the following result we study the regularity of $u$ and show that it solves \eqref{eq:33}.
\begin{lemma}\label{lem:regularity}
Let $Q\in L^\infty(\R^N)$, $\frac{2(N+1)}{N-1}\leq p\leq\frac{2N}{N-2}$ and consider a solution 
$u\in L^p(\R^N)$ of \eqref{eqn:integ12}. Then $u$ belongs to
$W^{2,q}(\R^N)\cap\cC^{1,\alpha}(\R^N)$ for all $p\leq q<\infty$,  $0<\alpha<1$,
and it is a strong solution of \eqref{eq:33}. Moreover, $u$ is the real part of a function
$\tilde{u}$ which satisfies Sommerfeld's outgoing radiation condition
\begin{equation}\label{eqn:sommerfeld}
\lim_{R\to\infty}\frac{1}{R} \int_{B_R}|\nabla \tilde
u(x)-i\tilde{u}(x)\hat x |^2\, dx=0,
\end{equation}
and its far field pattern is given by the
nonlinear relation
\begin{equation}\label{eqn:weak_asympt-r}
\lim_{R\to\infty}\frac{1}{R}\int\limits_{B_R}\Bigl|u(x)
+2\Bigl(\frac{2\pi}{|x|}\Bigr)^{\frac{N-1}{2}} \Re \bigl[e^{i|x|-\frac{i(N-1)\pi}{4}}g_u(\hat x)\bigr]\Bigr|^2\, dx=0
\end{equation}
with $g_u: S^{N-1}$ $\to$ $\C$,
$g_u(\xi)=-\frac{i}{4}(2\pi)^{\frac{2-N}{2}}\cF(Q|u|^{p-2}u)(\xi)$.
\end{lemma}

\begin{proof}
We start by proving that the solution $u$ is bounded. For this, we use a similar proof as in \cite[Lemma B.3]{struwe},
based on Moser's iteration technique.
Since $Q\in L^\infty(\R^N)$ and $\frac{2(N+1)}{N-1}\leq p\leq 2^\ast$, we first note that 
Proposition \ref{prop:appendix} gives $u\in W^{2,p'}_{\text{loc}}(\R^N)$ and for every $x_0\in\R^N$, 
$$
\|u\|_{W^{2,p'}(B_2(x_0))}\leq \tilde{C}\left( \|u\|_{L^p(\R^N)}+\|Q\|_{_\infty}\|u\|_{L^p(\R^N)}^{p-1}\right)
$$ 
with some constant $\tilde{C}>0$, independent of $x_0$. Moreover, $u$ is a strong solution of~\eqref{eq:33}. 
Using Sobolev's embedding theorem with the property $p'\geq \frac{2N}{N+2}$, 
we obtain $u\in W^{1,2}_{\text{loc}}(\R^N)$ with
$$
\|u\|_{W^{1,2}(B_2(x_0))}\leq \kappa \tilde{C}\left( \|u\|_{L^p(\R^N)}+\|Q\|_{_\infty}\|u\|_{L^p(\R^N)}^{p-1}\right) \quad\text{ for all }x_0\in\R^N,
$$
where the constant $\kappa$ is independent of $x_0$. Consider now $s>0$, $0<R\leq 2$, $L>0$ and 
a cut-off function $\eta\in\cC^\infty_c(\R^N)$ with $\text{supp }\eta\subset B_R(x_0)$. 
Testing~\eqref{eq:33} with the function $\varphi=\varphi_{s,L}=u\min\{|u|^{2s},L^2\}\eta^2$, we deduce that for every $x_0\in\R^N$,
\begin{align*}
&\int_{B_R(x_0)}Q|u|^p\min\{|u|^{2s},L^2\}\eta^2\, dx = \int_{B_R(x_0)}|\nabla u|^2\min\{|u|^{2s},L^2\}\eta^2\, dx\\
&\quad+2s\int_{\{|u|^s\leq L\}\cap B_R(x_0)}|\nabla u|^2|u|^{2s}\eta^2\, dx + 2\int_{B_R(x_0)}(\nabla u\cdot\nabla \eta)u\eta\min\{|u|^{2s}L^2\}\, dx\\
& \quad -\int_{B_R(x_0)}|u|^2\min\{|u|^{2s},L^2\}\eta^2\, dx\\
&\geq \frac12\int_{B_R(x_0)}|\nabla u|^2\min\{|u|^{2s},L^2\}\eta^2\, dx+2s\int_{\{|u|^s\leq L\}\cap B_R(x_0)}|\nabla u|^2|u|^{2s}\eta^2\, dx\\
&\quad-2\int_{B_R(x_0)}|u|^2|\nabla\eta|^2\min\{|u|^{2s},L^2\}\, dx-\int_{B_R(x_0)}|u|^2\min\{|u|^{2s},L^2\}\eta^2\, dx.
\end{align*}
Assuming that $u\in L^{2s+2}(B_R(x_0))$ holds with $M:=\sup\limits_{x_0\in\R^N}\|u\|_{L^{2s+2}(B_R(x_0))}<\infty$, we find
\begin{align*}
&\int_{B_R(x_0)}\bigl|\nabla(u\min\{|u|^s,L\}\eta)\bigr|^2\, dx\leq 3\int_{B_R(x_0)}|\nabla u|^2\min\{|u|^{2s},L^2\}\eta^2\, dx\\
&\quad + 3s^2\int_{\{|u|^s\leq L\}\cap B_R(x_0)}|\nabla u|^2|u|^{2s}\eta^2\, dx + 3\int_{B_R(x_0)}u^2|\nabla\eta|^2\min\{|u|^{2s},L^2\}\, dx\\
&\leq\max\{6,\frac32s\}\left(\int_{B_R(x_0)}\!\!\!\!\! Q|u|^p\min\{|u|^{2s},L^2\}\eta^2\, dx+3\|\nabla\eta\|_\infty^2M^{2s+2}+\|\eta\|_\infty^2 M^{2s+2}\right)\\
&\leq \max\{6,\frac32s\}\Bigl[\bigl(\|Q\|_\infty K^{p-2}\|\eta\|_\infty^2+3\|\nabla\eta\|_\infty^2+\|\eta\|_\infty^2\bigr)M^{2s+2} \\
&\quad \left.+\|Q\|_\infty\left(\int_{\{|u|\geq K\}\cap B_R(x_0)}\!\!\!\!|u|^{(p-2)\frac{N}{2}}\, dx\right)^{\frac2N}\left(\int_{B_R(x_0)}\Bigl|u\min\{|u|^s,L\}\eta\Bigr|^{2^\ast}\, dx\right)^{\frac{N-2}{N}}\right]\\
& \leq \max\{6,\frac32s\}\Bigl[\bigl(\|Q\|_\infty K^{p-2}\|\eta\|_\infty^2+3\|\nabla\eta\|_\infty^2+\|\eta\|_\infty^2\bigr)M^{2s+2} \\
&\quad +\kappa_2\|Q\|_\infty |B_R|^{\frac2p-\frac{(N-2)p}{N}}\Bigl(\!\!\!\int\limits_{\{|u|\geq K\}}\!\!\!\!\!\!|u|^p\, dx\Bigr)^{\frac{p-2}{p}}
\!\!\!\int_{B_R(x_0)}\bigl|\nabla(u\min\{|u|^s,L\}\eta)\bigr|^2\, dx\Bigr]
\end{align*}
for every $K\geq 0$ and where the constant $\kappa_2$ given by the Sobolev embedding $W^{1,2}(B_R(x_0))\embed L^{2^\ast}(B_R(x_0))$ 
is independent of $x_0$. Since $u\in L^p(\R^N)$, we can choose $K>0$ such that 
$$
\int_{\{|u|\geq K\}}|u|^p\, dx\leq \bigl(2\max\{6,\frac32s\}\kappa_2\|Q\|_\infty |B_R|^{\frac2p-\frac{(N-2)p}{N}}\bigr)^{-\frac{p}{p-2}},
$$
and consequently, setting $C=\max\{6,\frac32s\}\bigl(\|Q\|_\infty K^{p-2}\|\eta\|_\infty^2+3\|\nabla\eta\|_\infty^2+\|\eta\|_\infty^2\bigr)$, we find
$$
\int_{B_R(x_0)}\bigl|\nabla(u\min\{|u|^s,L\}\eta)\bigr|^2\, dx\leq 2CM^{2s+2}.
$$
Since the bound is uniform in $L$, we may let $L\to\infty$ to obtain
$$
\int_{B_R(x_0)}|\nabla(|u|^{s+1}\eta)|^2\, dx\leq 2CM^{2s+2},
$$
for all $x_0\in\R^N$, and by Sobolev's embedding theorem, $|u|^{s+1}\eta\in L^{2^\ast}(B_R(x_0))$. Since $\eta\in\cC^\infty_c(\R^N)$ was chosen arbitrarily with $\text{supp }\eta\subset B_R(x_0)$, we infer that $u\in L^{\frac{2N(s+1)}{N-2}}(B_r(x_0))$ for all $r<R$, and that
$\sup\limits_{x_0\in\R^N}\|u\|_{L^{\frac{2N(s+1)}{N-2}}(B_r(x_0))}<\infty$.

Consider now a strictly decreasing sequence of radii $(R_i)_{i\geq 0}$ such that $R_0=2$ and $R_i\geq 1$ for all $i$, 
and iterate the procedure above with $R=R_i$ and $s=s_i$, where $s_0=0$ and $s_i=\frac{Ns_{i-1}+2}{N-2}$, $i\geq 1$.
Since $\sup\limits_{x_0\in\R^N}\|u\|_{L^2(B_2(x_0))}<\infty$, we obtain that $u\in L^q(B_1(x_0))$ 
with $\sup\limits_{x_0\in\R^N}\|u\|_{L^q(B_1(x_0))}<\infty$ for all $1\leq q<\infty$.
Hence, the same is also true for $f=Q|u|^{p-2}u$ and Proposition \ref{prop:appendix}~(i) gives 
$u\in W^{2,q}_{\text{loc}}(\R^N)$ for all $1\leq q<\infty$, as well as the estimate
$$
\|u\|_{W^{2,N}(B_{\frac12}(x_0))}\leq \tilde{D}\left( \|u\|_{L^N(B_1(x_0))} 
+ \|Q\|_\infty \|u\|_{L^{N(p-1)}(B_1(x_0))}^{p-1}\right),
$$
for all $x_0\in\R^N$, where $\tilde{D}>0$ is independent of $x_0$. Using Sobolev's embedding theorem, we infer that 
$u\in L^\infty(B_{\frac12}(x_0))$ for all $x_0 \in \R^N$ with $\sup\limits_{x_0\in\R^N}\|u\|_{L^\infty(B_{\frac12}(x_0))}<\infty$, 
i.e. $u\in L^\infty(\R^N)$, as claimed. Applying Proposition \ref{prop:appendix} (ii) we then find 
$u\in W^{2,q}(\R^N)$ for every $p\leq q<\infty$ which gives the desired regularity for $u$.

Next we observe that $u=\Re(\tilde{u})$ for $\tilde{u}:=\cR
\bigl(Q|u|^{p-2}u\bigr)$, and we show that $\tilde u$ satisfies
(\ref{eqn:sommerfeld}). We already noted that (\ref{eq:24}) implies that 
\begin{equation}\label{eqn:a-h}
\lim_{R\to\infty}\frac1R\int_{B_R}\Bigl|\nabla [\cR f](x)
-i\hat x [\cR f](x) \Bigr|^2\, dx=0 \qquad \text{for all $f \in \cS$.}
\end{equation}
Let $(f_n)_n$ be a sequence in $\cS$ such that
$\|f_n-Q|u|^{p-2}u\|_{p'}\to 0$ as $n\to\infty$. As a consequence of Theorems 7 and 8 in \cite{gutierrez04}
and since $\frac{2N}{N+2}\leq p'\leq \frac{2(N+1)}{(N+3)}$, we then
have
$$
\sup_{R>1 }\frac1R\int_{B_R}\Bigl|\nabla [\cR(f_n-Q|u|^{p-2}u)](x) -i\hat x [\cR(f_n-Q|u|^{p-2}u)](x)\Bigr|^2\, dx
$$
$$
\leq C \|f_n-Q|u|^{p-2}u\|_{p'}^2\to 0,\quad\text{as $n\to\infty$.}
$$
Consequently, \eqref{eqn:a-h} also holds with $f=Q|u|^{p-2}u$ and therefore
$\tilde{u}=\cR (Q|u|^{p-2}u)$ satisfies (\ref{eqn:sommerfeld}).\\
Finally, the nonlinear relation (\ref{eqn:weak_asympt-r}) follows
from Proposition~\ref{prop:weak_farfield}
since (\ref{eqn:weak_asympt}) holds for $\tilde u = \cR f$ with $f:=
Q |u|^{p-2}u \in L^{p'}(\R^N)$. 
\end{proof}

The remainder of this section is devoted to the proof of Theorem \ref{sec:introduction-4}, 
which we recall in a slightly different formulation.
\begin{theorem}
\label{sec:nonv-thm-2}
Let $N=3$  or $N=4$, $\frac{3N-1}{N-1}<p\leq\frac{2N}{N-2}$,  $Q\in L^\infty(\R^N)$,  and consider a solution $u\in L^p(\R^N)$ of $u=\mR(Q|u|^{p-2}u)$. 
Then there exists a constant
$C>0$ such that $$|u(x)| \le C |x|^{\frac{1-N}{2}}\quad\text{ for all }x \in \R^N \setminus
\{0\}.$$ Moreover, as $|x|\to\infty$,
$$
u(x)= -2\Bigl(\frac{2\pi}{|x|}\Bigr)^{\frac{N-1}{2}} \Re \bigl[e^{i|x|-\frac{i(N-1)\pi}{4}}g_u(\hat x)\bigr]
+o(|x|^{\frac{1-N}{2}}),
$$
with $g_u$ as in Lemma~\ref{lem:regularity}.
\end{theorem}

\begin{proof}
We start by showing that 
\begin{equation}
  \label{eq:31-u-in-L-p-1}
f:=Q|u|^{p-2}u\in L^1(\R^N).
\end{equation}
Since $Q\in L^\infty(\R^N)$, it is enough to prove that $u\in L^{p-1}(\R^N)$.  Our proof of this property 
is based on the resolvent estimate given in \cite[Theorem 6]{gutierrez04}. This theorem implies that, for any pair of exponents $t,q$ satisfying 
\begin{equation}
  \label{eq:31-gutierrez-pair-condition}
\frac{1}{t} > \frac{N+1}{2N}, \qquad \frac{1}{q}< \frac{N-1}{2N} \qquad \text{and}\qquad \frac{2}{N+1} \,\le\, \frac{1}{t} - \frac{1}{q} \,\le\, \frac{2}{N},  
\end{equation}
the operator $\mR$ maps $L^t(\R^N)$ into $L^q(\R^N)$. Using these mapping properties, we first show the following:\\
\underline {\em Claim:} {\em There exists $\delta>0$ depending only on $p$ and $N$ with the property that, if $u \in L^s(\R^N)$ for some $s \in (p-1,p]$, then also 
$u \in L^{\tilde s}(\R^N)$ for $\tilde s:= \max \{p-1, s-\delta\}$.}\\[0.1cm]
To show this, we let $s \in (p-1,p]$ and put $t:= \frac{s}{p-1}$, so that $f \in L^t(\R^N)$. Since $t \in (1,p']$ and $p> \frac{3N-1}{N-1}>\frac{2N}{N-1}$, we have $\frac{1}{t} \ge \frac{1}{p'} >\frac{N+1}{2N}$, so the first condition in (\ref{eq:31-gutierrez-pair-condition}) is satisfied. We now distinguish the following cases:\\
\underline{\em Case 1:} $\frac{1}{t}-\frac{2}{N+1} < \frac{N-1}{2N}$. In this case, (\ref{eq:31-gutierrez-pair-condition}) is satisfied for $q \ge 1$ defined by 
$\frac{1}{q}=  \frac{1}{t}-\frac{2}{N+1}=\frac{p-1}{s}-\frac{2}{N+1}$ , whereas 
$$
\frac{1}{q}-\frac{1}{s} = \frac{p-2}{s}-\frac2{N+1} \ge \frac{p-2}{p}-\frac2{N+1}=: \delta_0>0,
$$
since $p > \frac{2(N+1)}{N-1}$. Hence, putting $\delta:= \frac{\delta_0}{1+\delta_0}>0$, we have $q \le s - \delta$ and thus 
$u \in L^q(\R^N) \cap L^p(\R^N) \subset L^{\tilde s}(\R^N)$. Hence the claim is true in this case.\\
\underline{\em Case 2:} $\frac{N-1}{2N} \le \frac{1}{t}-\frac{2}{N+1}$. In this case, since $p \le \frac{2N}{N-2}$ and thus $\frac{1}{p'}- \frac{2}{N} \le \frac{N+2}{2N}- \frac{2}{N} < \frac{N-1}{2N}$, we may choose $t_0 \in [t,p']$ such that 
$$
\frac{1}{t_0}-\frac{2}{N} < \frac{N-1}{2N} \le \frac{1}{t_0} - \frac{2}{N+1}.
$$
We then have $f \in L^{t_0}(\R^N)$, and (\ref{eq:31-gutierrez-pair-condition}) is satisfied for $t_0$ in place of $t$ and any $q$ satisfying $\frac{1}{t_0}-\frac{2}{N} \le \frac{1}{q} < \frac{N-1}{2N}$. 
Since $p> \frac{3N-1}{N-1}$ by assumption and therefore $\frac{1}{p-1} < \frac{N-1}{2N}$, we may choose $q$ such that $\frac{1}{p-1} \le \frac{1}{q}$. We then conclude that $u \in L^q(\R^N) \cap L^p(\R^N) \subset L^{p-1}(\R^N) \cap L^p(\R^N) \subset L^{\tilde s}(\R^N)$, and so the claim is also true in this case.\\ 
By iteration, the claim immediately gives $u \in L^{p-1}(\R^N)$, and so  (\ref{eq:31-u-in-L-p-1}) follows.

In order to prove the first assertion of the theorem, we apply
Lemma~\ref{lem:asympt_fct} with $K=\Psi$ and $V=Q|u|^{p-2}$.
Notice that 
from Lemma~\ref{lem:regularity} and the above, we know that $u\in L^{p-1}(\R^N)\cap L^\infty(\R^N)$ and
$Vu=f\in L^1(\R^N)\cap L^\infty(\R^N)$. Hence, there also holds $V\in L^{\frac{p-1}{p-2}}(\R^N)\cap L^\infty(\R^N)$
, and since $p>\frac{3N-1}{N-1}$, we see that $\frac{p-1}{p-2}<\frac{2N}{N+1}$. Thus, the conditions of
Lemma~\ref{lem:asympt_fct} are satisfied and the result follows.
The second assertion then follows from Proposition~\ref{prop:farfield_N} 
by remarking that
$$
|f(x)|\leq \|Q\|_\infty |u(x)|^{p-1}\leq C^{p-1}\|Q\|_\infty |x|^{\frac{(1-N)(p-1)}{2}}\quad\text{ for }x\neq 0,
$$
where $\frac{(p-1)(1-N)}{2}<-N$. The proof is therefore complete.
\end{proof}

\section{Existence of solutions in the compact case}\label{sec:compact}
We now assume, in addition to the assumptions of Section~\ref{sec:dual}, that $p<\frac{2N}{N-2}$ and $Q(x)\to 0$ as $|x|\to\infty$.
In this case, we shall prove the existence of infinitely many pairs
$\{\pm u\}$ of critical points for $J$ using a variant of the symmetric 
Mountain Pass Theorem \cite{ambrosetti-rabinowitz73}. 

For this purpose, we collect further properties of $\mK_p$ and 
the functional $J$.

\begin{lemma}\label{lem:MP_hyp-1}
For every $m \in \N$, there exists an $m$-dimensional subspace $\cW
\subset \cC_c^\infty(\R^N)$ with the following properties: 
\begin{itemize}
\item[(i)] $\int_{\R^N}v \mK_p v\, dx>0$ for all $v \in \cW \setminus
  \{0\}$.
\item[(ii)] There exists $R=R(\cW)>0$ such that $J(v) \le 0$ for every $v
  \in \cW$ with $\|v\|_{p'} \ge R$.
\end{itemize}
\end{lemma}

\begin{proof}
Since $Q \not \equiv 0$, there exists a point of density one for the set $\{Q >0\}$. 
Without loss of generality, we may assume that $x_0=0$. Then for 
$\delta>0$ sufficiently small we have 
\begin{equation}
  \label{eq:21}
|Q^{-1}(0) \cap B_\delta(0)| \le  \Bigl(\frac{1}{4m^2}\Bigr)^N |B_\delta(0)|.
\end{equation}
Let 
$$
\Psi^*(\tau):= \inf_{B_{\tau}(0)\setminus \{0\}}\Psi\qquad
\text{and}\qquad \Psi_*(\tau):=\|\Psi\|_{L^\infty(\R^N \setminus
  B_{\tau}(0))} \qquad \text{for $\tau >0$.}
$$
Since $\Psi$ is bounded outside of every neighborhood of zero and
$\Psi(x)|x|^{N-2}$ tends to a positive constant as $|x| \to 0$ by (\ref{eq:3}),
we may fix $\delta >0$ such that (\ref{eq:21}) holds and that
\begin{equation}
  \label{eq:19}
\Psi^*(\tau) >
(m-1)\Psi_*(m\tau) \qquad \text{for
  $\tau \in (0,\delta]$.}
\end{equation}
Moreover, it is easy to see that there exists $m$ disjoint open
balls $B^1,\dots,B^m \subset B_\delta(0)$ of diameter $\tau:= \frac{\delta}{m^2}$ such that 
$$
\dist(B^i,B^j):= \inf\{|x-y|\::\:x \in B^i, \:y \in B^j\} \ge \frac{\delta}{m}.
$$
Since $|B^i| = \Bigl(\frac{1}{2m^2}\Bigr)^N |B_\delta(0)|$ for $i=1,\dots,m$, we also have 
\begin{equation}
  \label{eq:22}
|B^i \cap \{Q>0\}|>0 \qquad \text{for $i=1,\dots,m$}  
\end{equation}
by (\ref{eq:21}). We now fix functions $z_i \in \cC_c^\infty(\R^N)$, $i=1,\dots,m$
such that $z_i>0$ in $B^i$ and $z_i \equiv 0$ in $\R^N \setminus B^i$.
Moreover, we let $\cW$ denote the span of $z_1,\dots,z_m$. Then
any $v \in \cW \setminus \{0\}$ can be written as $v= \sum \limits_{i=1}^m a_i
z_i$ with $a= (a_1,\dots,a_m) \in \R^m \setminus \{0\}$, and thus we
have 
\begin{align*}
 \int_{\R^N}& v \mK_p v\, dx=\sum_{i,j=1}^m
a_i a_j \int_{B^i}\int_{B^j} \Psi(x-y)Q(x)^\frac1p Q(y)^\frac1p z_i(x)z_j(y)\,
dxdy \\
\ge& \Psi^*(\tau)  \sum_{i=1}^m a_i^2 \Bigl(\int_{B^i} Q(x)^\frac1p
z_i(x)\,dx\Bigr)^2\\
 &- \Psi_*(m \tau) \sum_{\stackrel{i,j=1}{i \not= j}}^m |a_i| |a_j| \Bigl(\int_{B^i} Q(x)^\frac1p
z_i(x)\,dx\Bigr) \Bigl(\int_{B^j} Q(x)^\frac1p
z_j(x)\,dx\Bigr)\\
\ge& \Psi^*(\tau)  \sum_{i=1}^m a_i^2 \Bigl(\int_{B^i} Q(x)^\frac1p
z_i(x)\,dx\Bigr)^2 \\
&- \frac{\Psi_*(m\tau)}{2} \sum_{\stackrel{i,j=1}{i \not=
    j}}^m \Bigl[ a_i^2 \Bigl(\int_{B^i} Q(x)^\frac1p
z_i(x)\,dx\Bigr)^2 + a_j^2 \Bigl(\int_{B^j} Q(x)^\frac1p
z_j(x)\,dx\Bigr)^2 \Bigr]\\
&=   \sum_{i=1}^m \Bigl(\Psi^*(\tau) -
(m-1)\Psi_*(m \tau) \Bigr)a_i^2
\Bigl(\int_{B^i} Q(x)^\frac1p
z_i(x)\,dx\Bigr)^2 >0
\end{align*}
as a consequence of (\ref{eq:19}) and (\ref{eq:22}). This shows
(i). As a consequence of (i) and continuity, we have 
$$
m_\ccW:= \inf_{v \in \cW, \|v\|_{p'}=1}\: \int_{\R^N}v \mK_p v\, dx >0,
$$
Hence 
$$
J(v)= \frac{\|v\|_{p'}^{p'}}{p'}-\frac12\int_{\R^N}v\mK_p v\, dx 
\le  \|v\|_{p'}^{p'}\Bigl(\frac{1}{p'}- \frac12\|v\|_{p'}^{2-p'}
m_\ccW \Bigr)\qquad \text{for $v \in \cW$.}
$$
Thus (ii) follows with $R:= \Bigl(\frac{2}{m_\ccW p'}\Bigr)^{\frac{1}{2-p'}}$.
\end{proof}

\begin{lemma}
  \label{sec:exist-solut-comp}
$J$ satisfies the Palais-Smale condition in $L^{p'}(\R^N)$.
\end{lemma}

\begin{proof}
Let $(v_n)_n\subset L^{p'}(\R^N)$ be a Palais-Smale sequence.
According to Lemma \ref{lem:MP_geom2} (iii), $(v_n)_n$ is bounded in $L^{p'}(\R^N)$. Hence, up to a subsequence, we
may assume $v_n\weak v\in L^{p'}(\R^N)$. From the convexity of the function $t\mapsto |t|^{p'}$ we obtain
\begin{align*}
\frac{1}{p'}\|v\|_{p'}^{p'}-\frac{1}{p'}\|v_n\|^{p'}_{p'}&\geq \int_{\R^N}|v_n|^{p'-2}v_n(v-v_n)\, dx\\
&= J'(v_n)(v-v_n) + \int_{\R^N}v_n\mK_p(v-v_n)\, dx\to 0
\end{align*}
as $n\to\infty$, taking into account the symmetry and the compactness of the Birman-Schwinger operator
$\mK_p$ proven in Lemma \ref{lem:compact}.
Consequently, $\|v\|_{p'}\geq \limsup\limits_{n\to\infty}\|v_n\|_{p'}$. On the other hand, the
weak convergence $v_n\weak v$ implies $\|v\|_{p'}\leq\liminf\limits_{n\to\infty}\|v_n\|_{p'}$, which together
gives $\lim\limits_{n\to\infty}\|v_n\|_{p'}=\|v\|_{p'}$, and hence $v_n\to v$ strongly in $L^{p'}(\R^N)$, as claimed.
\end{proof}

Combining Lemmas \ref{lem:MP_geom2} and \ref{sec:exist-solut-comp} with the symmetric Mountain Pass Theorem e.g. in the
form of \cite[Corollary 7.23]{ghoussoub}, we obtain the existence
of a sequence of nontrivial pairs $\{\pm v_n\}$
of critical points of $J$ with $J(v_n) \to \infty$ and thus  $\|v_n\|_{p'} \to \infty$ as $n \to \infty$. Setting $u_n:= \mR(Q^{\frac{1}{p}} v_n)$, we then have $v_n= Q^{\frac{1}{p'}}|u_n|^{p-2}u_n$ for all $n \in \N$ and thus $\|u_n\|_p \to \infty$ as $n \to \infty$.  Summarizing and taking Lemma
\ref{lem:regularity} into account, we can thus state the following.
\begin{theorem}
Let $\frac{2(N+1)}{N-1}\leq p<\frac{2N}{N-2}$ and consider a
nonnegative function $Q \in L^\infty(\R^N)$, $Q\not\equiv 0$ 
such that $Q(x)\to 0$ as $|x|\to\infty$. Then problem \eqref{eq:33}, \eqref{eqn:weak_asympt-r} 
admits a
sequence of pairs $\pm u_n \in W^{2,q}(\R^N)\cap\cC^{1,\alpha}(\R^N)$, $p\leq q<\infty$, $0<\alpha<1$ of strong solutions such that $\|u_n\|_{p} \to \infty$ as $n \to \infty$.
\end{theorem}
\section{Existence in the periodic case}\label{sec:dual_periodic}
In this section, we treat the case where $Q\in L^\infty(\R^N)$ is $1$-periodic on $\R^N$, i.e. $Q(x+e_i)=Q(x)$ for all $x\in\R^N$ and all 
$1\leq i\leq N$, $\{e_1,\ldots, e_N\}\subset\R^N$ denoting the standard basis in $\R^N$.

We shall prove the existence of solutions using a dual variational approach as before. 
Considering the dual functional
$J$: $L^{p'}(\R^N)$ $\to$ $\R$ given by \eqref{eqn:energy_funct},
we already know that $J$ is of class $\cC^1$ on $L^{p'}(\R^N)$ and, according to Lemma \ref{lem:MP_geom2}, that it possesses the mountain-pass
geometry. However, since $\mK_p$ is not compact anymore, the Palais-Smale condition does not hold in general. Nevertheless,
we may define a mountain-pass level for $J$ by setting
$$
c:=\inf_{\gamma\in\Gamma}\max_{t\in[0,1]}J(\gamma(t))
$$
where $\Gamma=\{\gamma\in C([0,1],L^{p'}(\R^N))\, :\, \gamma(0)=0\text{ and }J(\gamma(1))<0\}$.
Remark that by Lemma \ref{lem:MP_geom2} there holds  $\Gamma\neq \varnothing$ and $c>0$.
Our purpose is to show that $c$ is a critical level of $J$. We start by proving the existence of some Palais-Smale sequence for $J$.

\begin{lemma}\label{lem:bd_PS_2}
There exists a bounded Palais-Smale sequence $(v_n)_n\subset L^{p'}(\R^N)$ for $J$ at level $c$.
\end{lemma}
\begin{proof}
Suppose by contradiction, that no Palais-Smale sequence for $J$ exists at level $c$. In that case there are $0<\eps<\frac{c}{2}$
and $\delta>0$ such that $\|J'(v)\|\geq \delta>0$ for all $v\in L^{p'}(\R^N)$ satisfying $|J(v)-c|\leq 2\eps$. According to the
deformation Lemma \cite[Lemma 2.3]{willem} we can therefore find a homotopy $\eta\in C([0,1]\times L^{p'}(\R^N),L^{p'}(\R^N))$ such that
\begin{itemize}
\item[(i)] $\eta(0,v)=v$ for all $v\in L^{p'}(\R^N)$,
\item[(ii)] $\eta(t,v)=v$ for all $t\in[0,1]$, $v\in L^{p'}(\R^N)$ for which $J(v)\notin[c-2\eps,c+2\eps]$,
\item[(iii)] $\eta(t,\cdot)$ is a homeomorphism of $L^{p'}(\R^N)$ for all $t\in[0,1]$,
\item[(iv)] $J(\eta(1,v))\leq c-\eps$ for all $v\in L^{p'}(\R^N)$ such that $J(v)\leq c+\eps$.
\end{itemize}
Choosing now $\gamma\in\Gamma$ such that $\max\limits_{t\in[0,1]}J(\gamma(t))\leq c+\eps$ and setting $\tilde{\gamma}(t):=\eta(1,\gamma(t))$ 
for all $t\in[0,1]$, we obtain from (ii) $\tilde{\gamma}(0)=0$ and $\tilde{\gamma}(1)=\gamma(1)$, which in turn implies 
$\tilde{\gamma}\in\Gamma$ and therefore $\max\limits_{t\in[0,1]}J(\tilde{\gamma}(t))\geq c$. On the other hand, it follows from (iv) 
that
$$
J(\tilde{\gamma}(t))=J(\eta(1,\gamma(t)))\leq c-\eps\quad\text{for all }t\in[0,1]
$$
which is a contradiction. Therefore, there must exist some Palais-Smale sequence $(v_n)_n\subset L^{p'}(\R^N)$ at level $c$ for $J$. Moreover,
by Lemma \ref{lem:MP_geom2} (iii), $(v_n)_n$ is a bounded sequence.
\end{proof}

\begin{theorem}\label{thm:exist2}
Let $\frac{2(N+1)}{N-1}< p<2^\ast$ and consider a nonnegative function $Q\in L^\infty(\R^N)$, $Q\not\equiv 0$ which is $\Z^N$-periodic on $\R^N$. Then
\eqref{eq:33}, \eqref{eqn:weak_asympt-r} has a nontrivial strong solution $u\in W^{2,q}(\R^N)\cap\cC^{1,\alpha}(\R^N)$,
$p\leq q<\infty$, $0<\alpha<1$.
\end{theorem}
\begin{proof}
Let $(v_n)_n$ and $v$ be as in Lemma \ref{lem:bd_PS_2}. Since $J(v_n)\to c>0$ and $J'(v_n)v_n\to 0$ as $n\to\infty$, we find
$$
\lim_{n\to\infty}\int_{\R^N}Q^{\frac1p}v_n\mR(Q^{\frac1p}v_n)\, dx=\frac{2p'}{(2-p')}\lim_{n\to\infty}\Bigl[J(v_n)-\frac1{p'}J'(v_n)v_n\Bigr]=\frac{2p'}{(2-p')}c>0.
$$
Since $Q\in L^\infty(\R^N)$, the sequence $(Q^{\frac1p}v_n)_n$ is
bounded and Theorem~\ref{lem:conc_comp} gives the existence of 
$R, \zeta>0$ and of a sequence $(x_n)_n\subset\R^N$ such that, up to a subsequence,
\begin{equation}\label{eqn:liminf}
\int_{B_R(x_n)}|v_n|^{p'}\, dx \geq \zeta\quad\text{for all }n.
\end{equation}
Note that we may assume (taking $R$ larger if necessary) that $x_n\in\Z^N$ holds for all $n$. 
Setting $w_n(x)=v_n(x+x_n)$, $x\in\R^N$, we find that $(w_n)_n\subset L^{p'}(\R^N)$ is a bounded sequence. 
Hence, up to a subsequence, 
$w_n\weakto w$ in $L^{p'}(\R^N)$. Moreover, $J(w_n)=J(v_n)$ and
$\|J'(w_n)\|=\|J'(v_n)\|$ for all $n$ by the periodicity of $Q$ and the 
translation equivariance of $\mR$. Next we show that 
\begin{equation}
  \label{eq:23}
1_{B_{R'}}|w_n|^{p'-2}w_n\to 1_{B_{R'}}|w|^{p'-2}w\quad\text{strongly
  in $L^{p}(B_{R'})$ for every $R'>0$.}
\end{equation}
To see this, fix $\phi \in\cC^\infty_c(B_{R'}) \subset
\cC^\infty_c(\R^N)$. 
Then for $n,m\in\N$ we have 
\begin{align*}
&\left|\int_{\R^N}\Bigl(|w_n|^{p'-2}w_n -|w_m|^{p'-2}w_m\Bigr) \varphi\, dx \right| \\
&\qquad =\left| J'(w_n)\varphi-J'(w_m)\varphi + \int_{B_{R'}}\varphi\mK_p(w_n-w_m)\, dx\right|\\
&\qquad \leq \|J'(w_n)-J'(w_m)\| \|\varphi\|_{p'} + \|1_{B_{R'}}\mK_p(w_n-w_m)\|_{p} \|\varphi\|_{p'}.
\end{align*}
Since $\cC^\infty_c(B_{R'}) \subset L^{p'}(B_{R'})$ is dense,
$\|J'(w_n)\|\to 0$ as $n \to \infty$ and since, according to Lemma \ref{lem:compact}, $1_{B_{R'}}\mK_p$ is a compact operator, 
we deduce that $|w_n|^{p'-2}w_n$ is a Cauchy sequence in $L^p(B_{R'})$, so
that $|w_n|^{p'-2}w_n \to \tilde w$ strongly in $L^p(B_{R'})$ for some
$\tilde w \in
L^p(B_{R'})$. Up to a subsequence, $|w_n|^{p'-2}w_n \to \tilde w$ and,
equivalently, $w_n \to |\tilde w|^{p-2}\tilde w$ pointwise a.e. on $B_{R'}$. The uniqueness of the weak limit
then gives $w=|\tilde w|^{p-2}\tilde w$, i.e. $\tilde w=|w|^{p'-2}w$
on $B_{R'}$. Hence (\ref{eq:23}) is true. As a consequence, 
$$
0<\zeta\leq \int_{B_R(x_n)}|v_n|^{p'}\, dx=\int_{B_R}|w_n|^{p'}\, dx \to \int_{B_R}|w|^{p'}\, dx\quad\text{as }n\to\infty,
$$
which implies $w\neq 0$. Next we show that $w$ is a critical
point of $J$. For every $\varphi\in\cC^\infty_c(\R^N)$ 
we have, by (\ref{eq:23}),
$$
\int_{\R^N}|w_n|^{p'-2}w_n\varphi\, dx \to \int_{\R^N}|w|^{p'-2}w\varphi\, dx\quad\text{as }n\to\infty
$$
and also, since $\mK_p$ is a bounded linear operator, 
$$
\int_{\R^N}\varphi\mK_p(w_n)\, dx\to \int_{\R^N}\varphi\mK_p(w)\, dx\quad\text{as }n\to\infty.
$$
Consequently,
\begin{align*}
J'(w)\varphi &= \int_{\R^N}|w|^{p'-2}w\varphi\, dx -
\int_{\R^N}\varphi\mK_p(w)\, dx\\
&= \lim_{n \to
  \infty}\Bigl(\int_{\R^N}|w_n|^{p'-2}w_n\varphi\, dx -
\int_{\R^N}\varphi\mK_p(w_n)\, dx\Bigr)= \lim_{n \to \infty}J'(w_n)\phi=0.
\end{align*}
Therefore, $w\in L^{p'}(\R^N)$ is a nontrivial critical point of $J$ and Lemma \ref{lem:regularity} concludes the proof.
\end{proof}

\section*{Acknowledgements}
The authors would like to than the referee for his/her careful reading of the manuscript and his/her suggestions which greatly simplified the proof of Theorem~\ref{sec:nonv-thm-2}.

\appendix 
\section{ }\label{sec:appendix}
\begin{proposition}\label{prop:appendix}
Let $\frac{2(N+1)}{N-1}\leq p\leq\frac{2N}{N-2}$ and $f\in
L^{p'}(\R^N)$. Then $u:=\cR  f \in W^{2,p'}_{\text{loc}}(\R^N)\cap L^p(\R^N)$ is a strong solution of $-\Delta u - u=f$ 
in $\R^N$. Moreover, for every $r>0$, there exists a constant $\tilde{C}>0$ depending only on $r$, $p$ and $N$, such that for all $x_0\in\R^N$,
\begin{equation}
\label{cald_zygm_2}
    \|u\|_{W^{2,p'}(B_r(x_0))}\leq \tilde{C}\left( \|u\|_{L^p(\R^N)}+\|f\|_{L^{p'}(\R^N)}\right).
\end{equation}
Furthermore,
\begin{itemize}
  \item[(i)] if $f\in L^{p'}(\R^N)\cap L^q_{\text{loc}}(\R^N)$ and
    $u\in L^q_{\text{loc}}(\R^N)$ for some $q \in (1,\infty)$, then 
   $u\in W^{2,q}_{\text{loc}}(\R^N)$, and for every $r>0$ there exists
   a constant $D>0$ depending only on $r$, $p$, $q$ and $N$, such that 
\begin{equation}\label{cald_zygm_3}
\|u\|_{W^{2,q}(B_r(x_0))}\leq D\left(
  \|u\|_{L^q(B_{2r}(x_0))}+\|f\|_{L^q(B_{2r}(x_0))}\right) 
\end{equation}
for all $x_0\in\R^N$.
  \item[(ii)] If $f\in L^{p'}(\R^N)\cap L^q(\R^N)$ and $u\in L^q(\R^N)$ for some $q \in (1,\infty)$, then $u\in W^{2,q}(\R^N)$.
\end{itemize}
\end{proposition}

\begin{proof}
We first show that 
\begin{equation}
  \label{eq:25}
\text{$-\Delta u -u = f$ in distributional sense.}  
\end{equation}
For this we first
assume that $f\in\cS$. In this case, $\cR f\in\cS'$ is given by
$$
\ps{\cR f}{\varphi}
=\lim\limits_{\eps\to0^+}\int_{\R^N}\varphi(x)\, \cF^{-1}\left(\frac{\widehat{f}(\cdot)}{|\cdot|^2-1-i
  \eps}\right)\, dx 
=\lim\limits_{\eps\to0^+}\int_{\R^N}\frac{\check{\varphi}(\xi)\widehat{f}(\xi)}{|\xi|^2-1-i
  \eps}\, d\xi
$$
for all $\varphi\in\cS$, where, as usual, $\check \phi$ is an abbreviation for $\cF^{-1}(\phi)$.
Hence, setting $u=\cR f$, we obtain for every $\varphi\in\cS$:
\begin{align*}
\ps{-\Delta u-u}{\varphi}&=
\ps{\cR f}{-\Delta\varphi-\varphi}
=\lim_{\eps\to 0^+}\int_{\R^N}\frac{\widehat{f}(\xi)\check
  \varphi(\xi)(|\xi|^2-1)}{|\xi|^2-1-i
  \eps}\, d\xi\\
&=\lim_{\eps\to 0^+}\int_{\R^N}\frac{\widehat{f}(\xi)\check
  \varphi(\xi)(|\xi|^2-1-i \eps)}{|\xi|^2-1-i
  \eps}\, d\xi
=\int_{\R^N}f(x)\varphi(x)\, dx=\ps{f}{\varphi}.
\end{align*}
Here we used the fact that 
$$
\lim_{\eps\to 0^+}\int_{\R^N}\frac{i \eps}{|\xi|^2-1-i \eps}g(\xi)\,
d\xi=0 \qquad \text{for every $g \in \cS$,}
$$
which follows from Lebesgue's Theorem since
$|\frac{i\eps}{|\xi|^2-1-i\eps}|\leq 1$ for every $\xi\in\R^N$,
$\eps>0$ and $\lim\limits_{\eps\to 0^+}\frac{i \eps}{|\xi|^2-1-i \eps}=0$ for $\xi
\in \R^N$ with $|\xi| \not = 1$. Hence we have $-\Delta u -u=f$ in the
distributional sense. Now let $f\in L^{p'}(\R^N)$ and consider a
sequence $(f_n)_n\subset\cS$ with $\|f_n- f\|_{p'}\to 0$ as
$n\to\infty$. Then 
 $u_n:=\cR f_n$ solves $-\Delta u_n-u_n=f_n$ in distributional sense,
 and $u_n \to u$ in $L^p(\R^N)$ by Theorem~\ref{thm:KRS}. 
Consequently, $-\Delta u_n- u_n\to f$ and $u_n\to u$ in $\cS'$ as
$n\to\infty$, so that (\ref{eq:25}) is true.

We now take $x_0\in\R^N$, $r>0$ and consider the mollification $(u_\eps)_{\eps>0}$ of $u:=\cR f$, i.e., $u_\eps:=\rho_\eps\ast u$ 
where $\rho_\eps(x)=\eps^{-N}\rho(\frac{x}{\eps})$, $x\in\R^N$ for some function $\rho\in\cC^\infty_c(\R^N)$ satisfying $\rho(x)\geq 0$, 
for all $x\in\R^N$, $\text{supp}(\rho)\subset B_1$ and $\int_{\R^N}\rho\, dx=1$. Since $u\in L^p(\R^N)$,
we obtain $u\in L^{p'}(B_r(x_0))$ and consequently, $u_\eps\to u$ in $L^{p'}(B_r(x_0))$ as $\eps\to 0^+$. Similarly, considering 
the mollification $(f_\eps)_{\eps>0}$ of $f$, we see that $f_\eps\to f$ in $L^{p'}(\R^N)$ and therefore also in $L^{p'}(B_r(x_0))$, 
as $\eps\to 0^+$. From the properties of the mollification of $L^p$-functions and of tempered distributions, with respect
to differential operators with constant coefficients (see \cite{rudin_fa}), we obtain
$$
-\Delta u_\eps- u_\eps=-\Delta(u\ast\rho_\eps)-(u\ast\rho_\eps)\\
=(-\Delta u - u)\ast\rho_\eps=f\ast\rho_\eps=f_\eps \quad \text{in $\R^N$.}
$$
Therefore, the elliptic regularity theory (see \cite[Theorem 9.11]{gilbarg-trudinger}) shows the existence, for all $r>0$, of some constant 
$C>0$, depending only on $r$, $p$ and $N$, such that
\begin{equation}\label{cald_zygm_1}
\|u_\eps\|_{W^{2,p'}(B_{r}(x_0))}\leq C\left( \|u_\eps\|_{L^{p'}(B_{2r}(x_0))}+\|f_\eps\|_{L^{p'}(B_{2r}(x_0))}\right) 
\quad\text{for all }\eps>0.
\end{equation}
Choosing some sequence $(\eps_n)_n\subset(0,\infty)$ such that $\eps_n\to 0$ as $n\to\infty$ and replacing $u_\eps$ by $u_{\eps_n}-u_{\eps_m}$ 
in \eqref{cald_zygm_1} gives that $(u_{\eps_n})_n$ is a Cauchy sequence in $W^{2,p'}(B_{r}(x_0))$ and therefore, there exists 
$w\in W^{2,p'}(B_r(x_0))$ such that $u_{\eps_n}\to w$ in $W^{2,p'}(B_r(x_0))$ as 
$n\to\infty$. Since this also implies $u_{\eps_n}\to w$ in $L^{p'}(B_r(x_0))$, we find that $w=u$ a.e. in $B_r(x_0)$, 
and it follows that $u\in W^{2,p'}(B_r(x_0))$ and $u$ solves the equation $-\Delta u -u=f$ almost everywhere in 
$B_r(x_0)$. Furthermore, \eqref{cald_zygm_1} gives
\begin{align*}
\|u\|_{W^{2,p'}(B_r(x_0))}&\leq C\left( \|u\|_{L^{p'}(B_{2r}(x_0))}+\|f\|_{L^{p'}(B_{2r}(x_0))}\right)\\
&\leq \tilde{C}\left( \|u\|_{L^p(\R^N)}+\|f\|_{L^{p'}(\R^N)}\right),
\end{align*}
where $\tilde{C}=C\max\{1,[\omega_N(2r)^N]^{\frac{p-2}{p}}\}$ and $\omega_N$ denotes the volume of the unit ball in $\R^N$. 
Since $r>0$ and $x_0\in\R^N$ were arbitrarily chosen, it follows that $u\in W^{2,p'}_{\text{loc}}(\R^N)$ is a strong solution of 
$-\Delta u-u=f$ and, for every $r>0$, there exists a constant
$\tilde{C}>0$ depending only on $r$, $p$ and $N$ such that
(\ref{cald_zygm_2}) holds for all $x_0\in\R^N$.\\
(i) Considering as before the mollifications $(u_\eps)_{\eps>0}$ of $u$ and $(f_\eps)_{\eps>0}$ of $f$, we obtain from the 
previous argument that $-\Delta u_\eps -u_\eps =f_\eps$ on
$\R^N$. Moreover, for all $x_0\in\R^N$ and $r>0$, $u_\eps\to u$ and 
$f_\eps\to f$ in $L^q(B_r(x_0))$ as $\eps\to 0^+$. Using again
elliptic regularity theory and reasoning as above, we find  $u\in W^{2,q}_{\text{loc}}(\R^N)$ and, for every $r>0$, 
the existence of some constant $D$, depending only on $r$, $p, q$ and
$N$ such that \eqref{cald_zygm_3} holds for all $x_0\in\R^N$. \\
(ii) 
As a consequence of (i), there holds $u\in W^{2,q}_{\text{loc}}(\R^N)$ and $u$ solves $-\Delta u -u=f$ a.e. on $\R^N$.
Considering again the mollifications  $(u_\eps)_{\eps>0}$ of $u$ and $(f_\eps)_{\eps>0}$ of $f$ we see, using the Calder\'on-Zygmund
estimate (see \cite[Corollary 9.10]{gilbarg-trudinger}), that for any sequence $(\eps_n)_n\subset(0,\infty)$ such that $\eps_n\to 0$
as $n\to\infty$, the sequence $(u_{\eps_n})_n$ is a Cauchy sequence in $W^{2,q}(\R^N)$. Since the argument in (i) implies 
$u_{\eps_n}\to u$ in $W^{2,q}_{\text{loc}}(\R^N)$, we conclude that $u\in W^{2,q}(\R^N)$.
\end{proof}

\bibliographystyle{plain}

\end{document}